\documentclass{amsart}
\usepackage{amsmath}
\usepackage{amsfonts}
\usepackage{amssymb}
\usepackage{graphicx}
\usepackage{pgf}
\usepackage{tikz}
\usetikzlibrary{arrows,automata}

\newtheorem{theorem}{Theorem}

\newtheorem{corollary}{Corollary}

\newtheorem{lemma}{Lemma}

\newtheorem{proposition}{Proposition}

\numberwithin{equation}{section}

\begin{document}
\title{Hankel determinants of the Cantor sequence}
\author{Zhi-Xiong WEN}
\address{School of Mathematics and Statistics\\
Huazhong University of Science and Technology\\
430074, Wuhan, P. R. China}
\email{zhi-xiong.wen@hust.edu.cn}
\author{Wen WU$^{*}$}
\address{Department of Mathematics\\
Hubei University\\
430062, Wuhan, P. R. China}
\email[Corresponding author]{hust.wuwen@gmail.com; wuwen@hubu.edu.cn}
\thanks{This work is supported by NSFC (Grant Nos. 11371156). \\ \indent $*$ Wen Wu is the corresponding author.}

\begin{abstract}
 In the paper, we give the recurrent equations of the Hankel determinants of the Cantor sequence, and show that the Hankel determinants as a double sequence is $3$-automatic. With the help of the Hankel determinants, we prove that the irrationality exponent of the Cantor number, i.e. the transcendental number with Cantor sequence as its $b$-ary expansion, equals $2$.
\end{abstract}
\maketitle

\section{Introduction.}
Let $\sigma$ be the \emph{Cantor substitution} defined over the alphabet $\mathcal{A}=\{a,b\}$ by
\[\sigma: a\mapsto aba, b\mapsto bbb.\]
Since the word $\sigma^n(a)$ is a prefix of $\sigma^{n+1}(a)$, i.e., $\sigma^{n+1}(a)=\sigma^{n}(a)w$ where
$w$ is a finite word over the alphabet $\mathcal{A}$, the sequence of words $(\sigma^{n}(a))_{n\in\mathbb{N}}$
converges to the infinite sequence
\[\mathbf{c}:=c_0c_1c_2\cdots\in \mathcal{A}^{\mathbb{N}},\]
called the \emph{Cantor sequence}. In this paper, we take $a=1$ and  $b=0$. Then the sequence $(c_n)_{n\geq 0}$ takes the value 1 if the $3$-ary expansion of $n$ contains no '1', and 0 otherwise(see \cite{ALL}, \cite{Fogg}). Here are the first few terms:
\begin{center}
\begin{tabular*}{9cm}{@{\extracolsep{\fill}}c|llllllllllr}

$n$ & 0 & 1 & 2 & 3 & 4 & 5 & 6 & 7 & 8 & $\cdots$\\
\hline
$c_n$ & 1 & 0 & 1 & 0 & 0 & 0 & 1 & 0 & 1 & $\cdots$ \\

\end{tabular*}
\end{center}
The Cantor sequence is an \emph{automatic sequence} (see \cite{ALL}), i.e., it can be generated by a finite automaton. In detail, the Cantor sequence can be recognized by the $3$-automaton in Figure \ref{fig:auto} in direct reading with the initial state $a$ and the output map $a\mapsto 1, b\mapsto 0$.
\begin{figure}[h]
\label{fig:auto}
\centering
\begin{tikzpicture}[shorten >=1pt,node distance=2cm,auto]
  \node[initial,state] (q_0)                    {$a/1$};
  \node[state]         (q_1) [right of=q_0] {$b/0$};
  \path[->] (q_0) edge [loop above] node {0,2} ()
                  edge              node {1}   (q_1)
            (q_1) edge [loop above] node {0,1,2} ();
\end{tikzpicture}
\caption{Automaton Generating the Cantor Sequence.}
\end{figure}
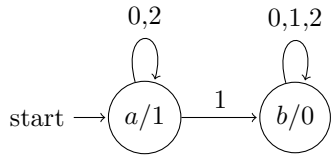
The Cantor sequence can be recognized as the discretization of the Cantor ternary set. In fact, we can construct the Cantor ternary set from the Cantor sequence as a recurrent set, following the way introduced by F. M. Dekking in \cite{Dek}.

\medskip

In this paper, we discuss some property of the Hankel determinants of the Cantor sequence.

\subsection{Hankel determinants of Cantor sequence.}
For a sequence of complex numbers $\mathbf{u}=u _{0}u_{1}\cdots u_{n}\cdots,$
the corresponding $(p,n)$-order \emph{Hankel matrix} $H_{n}^{p}$ is given
by
\begin{equation*}
H_{n}^{p}=%
\begin{pmatrix}
u_{p} & u_{p+1} & \cdots & u_{p+n-1} \\
u_{p+1} & u_{p+2} & \cdots & u_{p+n} \\
\cdots & \cdots & \cdots & \cdots \\
u_{p+n-1} & u_{p+n} & \cdots & u_{p+2n-2}%
\end{pmatrix}%
\end{equation*}%
where $n\geq 1$ and $p\geq 0$. And $|H_{n}^{p}|$ denote the determinant of
the matrix $H_{n}^{p}$.

Positive definiteness of Hankel matrices associated with a sequence are strong connected
the moment problem(see \cite{Sho}). The properties of Hankel determinants are also connnected
 to the combinatoric properties of the sequences and to the Pad\'{e} approximation(see \cite{Bak}, \cite{Bre}).
In \cite{Kam}, Kamae, Tamura and Wen studied the properties of Hankel determinants for the Fibonacci word and
give a quantitative relation between the Hankel determinant and the Pad\'{e} pair. Later, Tamura \cite{Tam} generalized the results
for a class of special sequences. Allouche, Peyri\`{e}re, Wen and
Wen studied the properties of Hankel determinants $|\mathcal{E}_n^p|$ of the Thue-Morse sequence in \cite{APWW}. They proved that
the Hankel determinants $|\mathcal{E}_n^p|$(modulo 2) recognized as a two-dimensional sequence (or \emph{double sequence})
 was 2-automatic. Recently, Gou, Wen and Wu \cite{GWW} proved that the Hankel determinants (modulo 2) of regular paper-folding sequence
are also periodic.
\medskip

Let $\Gamma_n^p$ be the $(p,n)$-order Hankel matrix of Cantor sequence. Our purpose
is to discuss the property of the two-dimensional sequence $\{|\Gamma_n^p|\}_{n,p\geq 0}$.
We have the following results.

\medskip
\noindent\textbf{Main Results.}
\emph{
\begin{enumerate}
\item For each $p\geq 0$, the sequence $\{|\Gamma_n^p|\}_{n\geq 0}$ is periodic(see Theorem $\ref{thm:perodic}$).
\item The two-dimensional sequence(modulo $3$) $\{|\Gamma_n^p|\}_{n,p\geq 0}$ is $3$-automatic (see Theorem $\ref{thm:auto}$).
\end{enumerate}
}
Figure \ref{fig:hankel} shows $|\Gamma_n^p|$ modulo $3$ for $1\leq n\leq 96$ and $0\leq p\leq 127$, where $0$'s are replaced by blue squares, $1$'s are replaced by green squares and $2$'s are replaced by red squares. Columns refer sequences $\{|\Gamma_n^p|\}_{n\geq 0}$.

\begin{figure}[htbp]
  \centering
  \includegraphics[width=\textwidth]{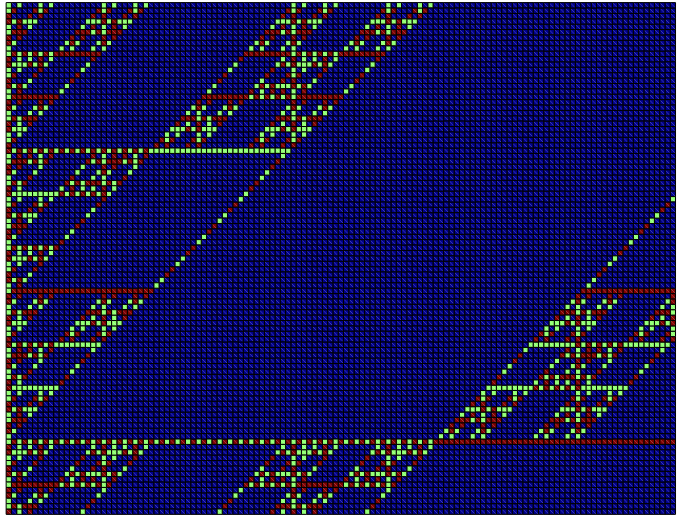}
  \vspace{-2cm}
  \caption{$|\Gamma^p_n|(\text{mod }3)$ for $1\leq n \leq 96$ and $0\leq p \leq 127$.}
	\label{fig:hankel}
\end{figure}

This article is organized as follows. Definitions and preliminaries are given in Section 2. In Section 3, we give recurrence formulae of Hankel determinants (modulo $3$) of the Cantor sequence. Section 4 is mainly devoted to characterize the periodicity and automaticity properties of Hankel determinants (modulo $3$) of the Cantor sequence. In Section 5, we will prove that the Pad\'{e} approximants to the generating series of the Cantor sequence exist and the irrationality exponent of the Cantor number and the difference sequence are both equal to $2$.

\section{Preliminaries.}

Assume that $M$ is an $n\times n$-matrix. Denote by $M^t$ the transpose of $%
M $. Let $M^{(i)}$ be the $n\times (n-1)$-matrix obtained by deleting the $i$%
-th \emph{column} of $M$, and $M_{(i)}$ be the $(n-1)\times n$ -matrix
obtained by deleting the $i$-th \emph{row} of $M$. $|M|$ denote the determinant of
the matrix $M$ and $\mathbf{0}_{m\times n}$ denote the $m\times n$ zero matrix.

It is easy to check that the Cantor sequence $\mathbf{c}=c_{0}c_{1}\cdots c_{n}\cdots \in \{0,1\}^{\mathbb{N}}$
can also be defined by the following recurrence equations:
\begin{equation}
c_{0}=1,~c_{3n}=c_{n},~c_{3n+1}=0,~c_{3n+2}=c_{n},\text{ for }n\geq 0.
\label{cantor}
\end{equation}%
For further consideration, we need another sequence $\mathbf{d}$ defined by $%
d_{n}=c_{n}+c_{n+2}$, and simultaneously we have
\begin{equation}
d_{3n}=2c_{n},~d_{3n+1}=c_{n+1},~d_{3n+2}=c_{n}.  \label{delta}
\end{equation}%
The Hankel matrices of this difference sequences $\mathbf{d}$ is denoted by
$\Delta_{n}^{p}$. An easy observation show that
\begin{equation*}
\Delta _{n}^{p}=\Gamma_{n}^{p}+\Gamma_{n}^{p+2}.
\end{equation*}

Let $K_{n}^{p}:=(u_{p+3(i+j-2)})_{1\leq i,j\leq n}$. When $\mathbf{u}=%
\mathbf{c}$ is the Cantor sequence, by (\ref{cantor}), we have for all $n\geq
1,p\geq 0,$
\begin{equation}
K_{n}^{3p}=\Gamma_{n}^{p},~K_{n}^{3p+1}=\mathbf{0}_{n\times
n},~K_{n}^{3p+2}=\Gamma_{n}^{p}.  \label{kc}
\end{equation}%
When $\mathbf{u}=\mathbf{d }$, by (\ref{delta}), for all $n\geq 1,p\geq 0,$
\begin{equation}
K_{n}^{3p}=2\Gamma_{n}^{p}\equiv
-\Gamma_{n}^{p},~K_{n}^{3p+1}=\Gamma_{n}^{p+1},~K_{n}^{3p+2}=\Gamma_{n}^{p},
\label{kd}
\end{equation}%
where the symbol $\equiv $, unless otherwise stated, means equality modulo $%
3 $ throughout this paper.

Now we will give an auxiliary lemma on matrices.

\begin{lemma}
\label{lem:matirx} For all $n\geq 2,p\geq 0$,
\begin{equation*}
\left\vert
\begin{matrix}
\Gamma_{n}^{p} & \Gamma_{n}^{p+1} \\
\Gamma_{n}^{p+1} & -\Gamma_{n}^{p}%
\end{matrix}%
\right\vert =(-1)^{n}|\Gamma_{n}^{p}|\cdot |\Delta_{n}^{p}|+(-1)^{n+1}|\Gamma
_{n+1}^{p}|\cdot |\Delta _{n-1}^{p}|.
\end{equation*}
\end{lemma}

\begin{proof}
Suppose $\alpha _{p}^{n}$ is the column vector of the form $%
(c_{p},c_{p+1},\cdots ,c_{p+n-1})^{t}$, then
\begin{equation*}
\left\vert
\begin{matrix}
\Gamma_{n}^{p} & \Gamma_{n}^{p+1} \\
\Gamma_{n}^{p+1} & -\Gamma_{n}^{p}%
\end{matrix}%
\right\vert =\left\vert
\begin{matrix}
\Gamma_{n}^{p} & \alpha_{p+n}^{n} & \alpha _{p}^{n} & \Gamma_{n}^{p+1} \\
\mathbf{0}_{1\times n} & 1 & 0 & \mathbf{0}_{1\times n} \\
\mathbf{0}_{1\times n} & 0 & 1 & \mathbf{0}_{1\times n} \\
\Gamma_{n}^{p+1} & \alpha_{p+n+1}^{n} & \mathbf{0}_{n\times 1} & -\Gamma
_{n}^{p}%
\end{matrix}%
\right\vert .
\end{equation*}%
Note that $(\Gamma_{n}^{p}~\alpha _{p+n}^{n})=(\alpha _{p}^{n}~\Gamma
_{n}^{p+1})$, we have
\begin{eqnarray*}
\left\vert
\begin{matrix}
\Gamma_{n}^{p} & \Gamma_{n}^{p+1} \\
\Gamma_{n}^{p+1} & -\Gamma_{n}^{p}%
\end{matrix}%
\right\vert  &=&\left\vert
\begin{matrix}
\Gamma_{n}^{p} & \alpha_{p+n}^{n} & 0 & \mathbf{0}_{n\times n} \\
\mathbf{0}_{1\times n} & 1 & 0 & 0\cdots 0~(-1) \\
\mathbf{0}_{1\times n} & 0 & 1 & 0\cdots 0~~~~0~~ \\
\Gamma_{n}^{p+1} & \alpha_{p+n+1}^{n} & -\alpha_{p+1}^{n} & -\Delta
_{n}^{p}%
\end{matrix}%
\right\vert  \\
&=&\left\vert
\begin{matrix}
1 & 0 \\
0 & 1%
\end{matrix}%
\right\vert \cdot \left\vert
\begin{matrix}
\Gamma _{n}^{p} & \mathbf{0} \\
\Gamma _{n}^{p+1} & -\Delta _{n}^{p}%
\end{matrix}%
\right\vert  \\
&&+(-1)^{n+1}\left\vert
\begin{matrix}
0 & -1 \\
1 & 0%
\end{matrix}%
\right\vert \left\vert
\begin{matrix}
(\Gamma _{n+1}^{p})_{(n+1)} & \mathbf{0}_{n\times (n-1)} \\
(\Gamma _{n+1}^{p+1})_{(n+1)} & -(\Delta _{n}^{p})^{(n)}%
\end{matrix}%
\right\vert .
\end{eqnarray*}

Since
\begin{eqnarray*}
\left\vert
\begin{matrix}
(\Gamma _{n+1}^{p})_{(n+1)} & \mathbf{0}_{n\times (n-1)} \\
(\Gamma _{n+1}^{p+1})_{(n+1)} & -(\Delta _{n}^{p})^{(n)}%
\end{matrix}%
\right\vert  &=&\left\vert
\begin{matrix}
(\Gamma _{n+1}^{p})_{(n+1)} & \mathbf{0}_{n\times (n-1)} \\
\begin{array}{c}
\mathbf{0}_{(n-1)\times n} \\
(\alpha _{p+n+1}^{n+1})^{t}%
\end{array}
& -(\Delta _{n}^{p})^{(n)}%
\end{matrix}%
\right\vert  \\
&=&(-1)^{(n+1)}|\Gamma _{n+1}^{p}|\cdot |-\Delta _{n-1}^{p}| \\
&=&|\Gamma _{n+1}^{p}|\cdot |\Delta _{n-1}^{p}|,
\end{eqnarray*}%
then
\begin{equation*}
\left\vert
\begin{matrix}
\Gamma _{n}^{p} & \Gamma _{n}^{p+1} \\
\Gamma _{n}^{p+1} & -\Gamma _{n}^{p}%
\end{matrix}%
\right\vert =(-1)^{n}|\Gamma _{n}^{p}|\cdot |\Delta
_{n}^{p}|+(-1)^{n+1}|\Gamma _{n+1}^{p}|\cdot |\Delta _{n-1}^{p}|.
\end{equation*}
\end{proof}

Let
\begin{equation*}
P(n)=(e_{1},e_{4},\cdots ,e_{3n_{1}-2},e_{2},e_{5},\cdots
,e_{3n_{2}-1},e_{3},e_{6},\cdots ,e_{3n_{3}}),
\end{equation*}%
where $n_{1}=[\frac{n+2}{3}],n_{2}=[\frac{n+1}{3}],n_{3}=[\frac{n}{3}]$ and $%
e_{j}$ is the $j$-th unit column vector of order $n$, i.e., the column
vector with $1$ as the $j$-th entry and zeros elsewhere. And $|P(n)|=\pm 1$. For simplicity, we
write $P$ instead of $P(n)$, when no confusion can occur. When consider $%
P(3n),P(3n+1),P(3n+2)$, the following diagram shows $n_{1},n_{2}$ and $n_{3}$
in these cases:
\begin{equation*}
\begin{array}{c|ccc}
& n_{1} & n_{2} & n_{3} \\ \hline
3n & n & n & n \\
3n+1 & n+1 & n & n \\
3n+2 & n+1 & n+1 & n%
\end{array}%
\end{equation*}

Suppose $M=(m_{i,j})_{1\leq i,j\leq n}$ is an $n\times n$ matrix, then
\begin{equation*}
P^{t}MP=\left(
\begin{matrix}
(m_{3i-2,3j-2})_{n_{1}\times n_{1}} & (m_{3i-2,3j-1})_{n_{1}\times n_{2}} &
(m_{3i-2,3j})_{n_{1}\times n_{3}} \\
(m_{3i-1,3j-2})_{n_{2}\times n_{1}} & (m_{3i-1,3j-1})_{n_{2}\times n_{2}} &
(m_{3i-1,3j})_{n_{2}\times n_{3}} \\
(m_{3i,3j-2})_{n_{3}\times n_{1}} & (m_{3i,3j-1})_{n_{3}\times n_{2}} &
(m_{3i,3j})_{n_{3}\times n_{3}}%
\end{matrix}%
\right) ,
\end{equation*}%
where $(m_{3i-2,3j-1})_{s\times t}$ means the matrix $(m_{3i-2,3j-1})_{1\leq
i\leq s,1\leq j\leq t}$.

\medskip
When $M=H_{3n}^{p},~H_{3n+1}^{p}$ and $H_{3n+2}^{p}$, we have
\begin{eqnarray}
& & P^{t}H_{3n}^{p}P  \notag \\
&=&
\begin{pmatrix}
(U_{p+3(i+j-2)})_{n\times n} & (U_{p+3(i+j-2)+1})_{n\times n} &
(U_{p+3(i+j-2)+2})_{n\times n} \\
(U_{p+3(i+j-2)+1})_{n\times n} & (U_{p+3(i+j-2)+2})_{n\times n} &
(U_{p+3(i+j-2)+3})_{n\times n} \\
(U_{p+3(i+j-2)+2})_{n\times n} & (U_{p+3(i+j-2)+3})_{n\times n} &
(U_{p+3(i+j-2)+4})_{n\times n}%
\end{pmatrix}
\notag \\
&=&
\begin{pmatrix}
K_{n}^{p} & K_{n}^{p+1} & K_{n}^{p+2} \\
K_{n}^{p+1} & K_{n}^{p+2} & K_{n}^{p+3} \\
K_{n}^{p+2} & K_{n}^{p+3} & K_{n}^{p+4}%
\end{pmatrix}%
,  \label{h3n}
\end{eqnarray}
\begin{eqnarray}
P^{t}H_{3n+1}^{p}P &=&
\begin{pmatrix}
K_{n+1}^{p} & (K_{n+1}^{p+1})^{(n+1)} & (K_{n+1}^{p+2})^{(n+1)} \\
(K_{n+1}^{p+1})_{(n+1)} & K_{n}^{p+2} & K_{n}^{p+3} \\
(K_{n+1}^{p+2})_{(n+1)} & K_{n}^{p+3} & K_{n}^{p+4}%
\end{pmatrix}%
,  \label{h3n+1} \\
P^{t}H_{3n+2}^{p}P &=&
\begin{pmatrix}
K_{n+1}^{p} & K_{n+1}^{p+1} & (K_{n+1}^{p+2})^{(n+1)} \\
K_{n+1}^{p+1} & K_{n+1}^{p+2} & (K_{n+1}^{p+3})^{(n+1)} \\
(K_{n+1}^{p+2})_{(n+1)} & (K_{n+1}^{p+3})_{(n+1)} & K_{n}^{p+4}%
\end{pmatrix}%
,  \label{h3n+2}
\end{eqnarray}

\section{Recurrent equations}
In this section, we establish the recurrence formulae for the sequence $|\Gamma_n^p|(n\geq 2, p\geq 0)$, which is the key result in this paper. Through these eighteen recurrence formulae, we can evaluate all the Hankel determinants $|\Gamma_n^p|,|\Delta_n^p|(n\geq 2, p\geq 0)$.
\begin{theorem}
\label{mainthm} For $p\geq 0$ and $n\geq 2$, one has

\begin{enumerate}
\item $|\Gamma _{3n}^{3p}|=(-1)^{n}|\Gamma _{n}^{p}|^{2}\cdot |\Delta
_{n}^{p}|+(-1)^{n+1}|\Gamma _{n}^{p}|\cdot |\Gamma _{n+1}^{p}|\cdot |\Delta
_{n-1}^{p}|,$

\item $|\Gamma _{3n+1}^{3p}|=(-1)^{n}|\Gamma _{n}^{p}|\cdot |\Gamma
_{n+1}^{p}|\cdot |\Delta _{n}^{p}|+(-1)^{n+1}|\Gamma _{n+1}^{p}|^{2}\cdot
|\Delta _{n-1}^{p}|,$

\item $|\Gamma _{3n+2}^{3p}|=(-1)^{n}|\Gamma _{n+1}^{p}|^{2}\cdot |\Delta
_{n}^{p}|,$

\item $|\Gamma _{3n}^{3p+1}|=(-1)^{n}|\Gamma _{n}^{p}|\cdot |\Gamma
_{n}^{p+1}|\cdot |\Delta _{n}^{p}|+(-1)^{n+1}|\Gamma _{n}^{p+1}|\cdot
|\Gamma _{n+1}^{p}|\cdot |\Delta _{n-1}^{p}|,$

\item $|\Gamma _{3n+1}^{3p+1}|=(-1)^{n+1}|\Gamma _{n+1}^{p}|^{2}\cdot
|\Delta _{n-1}^{p+1}|,$

\item $|\Gamma _{3n+2}^{3p+1}|=(-1)^{n+1}|\Gamma _{n+1}^{p}|^{2}\cdot
|\Delta _{n}^{p+1}|,$

\item $|\Gamma _{3n}^{3p+2}|=(-1)^{n}|\Gamma _{n}^{p+1}|^{2}\cdot |\Delta
_{n}^{p}|,$

\item $|\Gamma _{3n+1}^{3p+2}|=(-1)^{n}|\Gamma _{n}^{p+1}|\cdot |\Gamma
_{n+1}^{p}|\cdot |\Delta _{n}^{p+1}|+(-1)^{n+1}|\Gamma _{n+1}^{p}|\cdot
|\Gamma _{n+1}^{p+1}|\cdot |\Delta _{n-1}^{p+1}|,$

\item $|\Gamma _{3n+2}^{3p+2}|=(-1)^{n+1}|\Gamma _{n+1}^{p+1}|^{2}\cdot
|\Delta _{n}^{p}|,$

\item $|\Delta _{3n}^{3p}|\equiv (-1)^{n}|\Gamma _{n}^{p}|\cdot |\Delta
_{n}^{p}|^{2}+(-1)^{n+1}|\Gamma _{n+1}^{p}|\cdot |\Delta _{n-1}^{p}|\cdot
|\Delta _{n}^{p}|,$

\item $|\Delta _{3n+1}^{3p}|\equiv (-1)^{n+1}|\Gamma _{n+1}^{p}|\cdot
|\Delta _{n}^{p}|^{2},$

\item $|\Delta _{3n+2}^{3p}|\equiv (-1)^{n}|\Gamma _{n+2}^{p}|\cdot |\Delta
_{n}^{p}|^{2}+(-1)^{n+1}|\Gamma _{n+1}^{p}|\cdot |\Delta _{n}^{p}|\cdot
|\Delta _{n+1}^{p}|,$

\item $|\Delta _{3n}^{3p+1}|\equiv (-1)^{n}|\Gamma _{n}^{p+1}|\cdot |\Delta
_{n}^{p}|^{2},$

\item $|\Delta _{3n+1}^{3p+1}|\equiv (-1)^{n}|\Gamma _{n+1}^{p+1}|\cdot
|\Delta _{n}^{p}|^{2},$

\item $|\Delta _{3n+2}^{3p+1}|\equiv (-1)^{n}|\Gamma _{n+2}^{p}|\cdot
|\Delta _{n}^{p}|\cdot |\Delta _{n}^{p+1}|+(-1)^{n+1}|\Gamma
_{n+1}^{p}|\cdot |\Delta _{n}^{p+1}|\cdot |\Delta _{n+1}^{p}|,$

\item $|\Delta _{3n}^{3p+2}|\equiv (-1)^{n}|\Gamma _{n}^{p+1}|\cdot |\Delta
_{n}^{p}|\cdot |\Delta _{n}^{p+1}|+(-1)^{n+1}|\Gamma _{n+1}^{p+1}|\cdot
|\Delta _{n}^{p}|\cdot |\Delta _{n-1}^{p+1}|,$

\item $|\Delta _{3n+1}^{3p+2}|\equiv (-1)^{n}|\Gamma _{n+1}^{p}|\cdot
|\Delta _{n}^{p+1}|^{2},$

\item $|\Delta _{3n+2}^{3p+2}|\equiv (-1)^{n}|\Gamma _{n+2}^{p}|\cdot
|\Delta _{n}^{p+1}|^{2}.$
\end{enumerate}
\end{theorem}

\begin{proof}
$1)$ Combine (\ref{kc}) and (\ref{h3n}), we have
\begin{eqnarray*}
|P^{t}\Gamma _{3n}^{3p}P| &=&\left\vert
\begin{matrix}
\Gamma _{n}^{p} & \mathbf{0}_{n\times n} & \Gamma _{n}^{p} \\
\mathbf{0}_{n\times n} & \Gamma _{n}^{p} & \Gamma _{n}^{p+1} \\
\Gamma _{n}^{p} & \Gamma _{n}^{p+1} & \mathbf{0}_{n\times n}%
\end{matrix}%
\right\vert  \\
&=&\left\vert
\begin{matrix}
\Gamma _{n}^{p} & \mathbf{0}_{n\times n} & \mathbf{0}_{n\times n} \\
\mathbf{0}_{n\times n} & \Gamma _{n}^{p} & \Gamma _{n}^{p+1} \\
\mathbf{0}_{n\times n} & \Gamma _{n}^{p+1} & -\Gamma _{n}^{p}%
\end{matrix}%
\right\vert .
\end{eqnarray*}%
Hence
\begin{equation*}
|\Gamma _{3n}^{3p}|=|P^{t}\Gamma _{3n}^{3p}P|=|\Gamma _{n}^{p}|\cdot
\left\vert
\begin{matrix}
\Gamma _{n}^{p} & \Gamma _{n}^{p+1} \\
\Gamma _{n}^{p+1} & -\Gamma _{n}^{p}%
\end{matrix}%
\right\vert .
\end{equation*}%
By Lemma \ref{lem:matirx},
\begin{equation*}
|\Gamma _{3n}^{3p}|=(-1)^{n}|\Gamma _{n}^{p}|^{2}\cdot |\Delta
_{n}^{p}|+(-1)^{n+1}|\Gamma _{n}^{p}|\cdot |\Gamma _{n+1}^{p}|\cdot |\Delta
_{n-1}^{p}|.
\end{equation*}

$2)$ Combine (\ref{kc}) and (\ref{h3n+1}), we have
\begin{eqnarray*}
|P^{t}\Gamma _{3n+1}^{3p}P| &=&\left\vert
\begin{matrix}
\Gamma _{n+1}^{p} & \mathbf{0}_{{n+1}\times n} & (\Gamma _{n+1}^{p})^{(n+1)}
\\
\mathbf{0}_{n\times {n+1}} & \Gamma _{n}^{p} & \Gamma _{n}^{p+1} \\
(\Gamma _{n+1}^{p})_{(n+1)} & \Gamma _{n}^{p+1} & \mathbf{0}_{n\times n}%
\end{matrix}%
\right\vert  \\
&=&\left\vert
\begin{matrix}
\Gamma _{n+1}^{p} & \mathbf{0}_{{n+1}\times n} & \mathbf{0}_{{n+1}\times n}
\\
\mathbf{0}_{n\times {n+1}} & \Gamma _{n}^{p} & \Gamma _{n}^{p+1} \\
\mathbf{0}_{n\times {n+1}} & \Gamma _{n}^{p+1} & -\Gamma _{n}^{p}%
\end{matrix}%
\right\vert .
\end{eqnarray*}%
By Lemma \ref{lem:matirx},
\begin{eqnarray*}
|\Gamma _{3n+1}^{3p}| &=&|\Gamma _{n+1}^{p}|\cdot \left\vert
\begin{matrix}
\Gamma _{n}^{p} & \Gamma _{n}^{p+1} \\
\Gamma _{n}^{p+1} & -\Gamma _{n}^{p}%
\end{matrix}%
\right\vert  \\
&=&(-1)^{n}|\Gamma _{n}^{p}|\cdot |\Gamma _{n+1}^{p}|\cdot |\Delta
_{n}^{p}|+(-1)^{n+1}|\Gamma _{n+1}^{p}|^{2}\cdot |\Delta _{n-1}^{p}|.
\end{eqnarray*}

$3)$ Combine (\ref{kc}) and (\ref{h3n+2}), we have
\begin{eqnarray*}
|P^{t}\Gamma _{3n+2}^{3p}P| &=&\left\vert
\begin{matrix}
\Gamma _{n+1}^{p} & \mathbf{0}_{(n+1)\times (n+1)} & (\Gamma
_{n+1}^{p})^{(n+1)} \\
\mathbf{0}_{(n+1)\times (n+1)} & \Gamma _{n+1}^{p} & (\Gamma
_{n+1}^{p+1})^{(n+1)} \\
(\Gamma _{n+1}^{p})_{(n+1)} & (\Gamma _{n+1}^{p+1})_{(n+1)} & \mathbf{0}%
_{n\times n}%
\end{matrix}%
\right\vert  \\
&=&\left\vert
\begin{matrix}
\Gamma _{n+1}^{p} & \mathbf{0}_{(n+1)\times (n+1)} & \mathbf{0}_{(n+1)\times
n} \\
\mathbf{0}_{(n+1)\times (n+1)} & \Gamma _{n+1}^{p} & (\Gamma
_{n+1}^{p+1})^{(n+1)} \\
(\Gamma _{n+1}^{p})_{(n+1)} & (\Gamma _{n+1}^{p+1})_{(n+1)} & -\Gamma
_{n}^{p}%
\end{matrix}%
\right\vert  \\
&=&\left\vert
\begin{matrix}
\Gamma _{n+1}^{p} & \mathbf{0}_{(n+1)\times (n+1)} & \mathbf{0}_{(n+1)\times
n} \\
\mathbf{0}_{(n+1)\times (n+1)} & \Gamma _{n+1}^{p} & \mathbf{0}_{(n+1)\times
n} \\
(\Gamma _{n+1}^{p})_{(n+1)} & (\Gamma _{n+1}^{p+1})_{(n+1)} & -\Gamma
_{n}^{p}-\Gamma _{n}^{p+2}%
\end{matrix}%
\right\vert .
\end{eqnarray*}%
Therefore,
\begin{equation*}
|\Gamma _{3n+2}^{3p}|=|\Gamma _{n+1}^{p}|^{2}\cdot |-\Gamma _{n}^{p}-\Gamma
_{n}^{p+2}|=(-1)^{n}|\Gamma _{n+1}^{p}|^{2}\cdot |\Delta _{n}^{p}|.
\end{equation*}

Formulae $(4)$ to $(18)$ can be proved using similar computation. We state the proof in the appendix.
\end{proof}

Now, we will extend those eighteen recurrent formulae for all $n,p\geq 0$.
\begin{proposition}
\label{prop:1}Define $|\Delta _{0}^{p}|=1$ for $p\geq 0,$ and
\begin{equation*}
|\Gamma _{0}^{p}|=\left\{
\begin{array}{cc}
2 & \text{if }p=0 \\
1 & \text{if }p\geq 1%
\end{array}%
\right. ,|\Delta _{-1}^{p}|=\left\{
\begin{array}{cc}
1 & \text{if }p=0 \\
0 & \text{if }p\geq 1%
\end{array}%
\right. .
\end{equation*}%
Then formulae of Theorem \ref{mainthm} holds for $p\geq 0$ and $n\geq 0$.
\end{proposition}

\begin{proof}
Using (\ref{cantor}), (\ref{delta}) and the facts: $|\Gamma _{1}^{p}|=c_{p},$
and $|\Delta _{1}^{p}|=d_{p}.$ We can check the formulae of Theorem \ref%
{mainthm} one by one.
\end{proof}

\section{Periodicity and automaticity properties.}
The periodicity and automaticity properties of the Hankel determinants
$|\Gamma_n^p|$, $|\Delta_n^p|$ $(n,p\geq 0)$ are discussed in this section.
By the recurrent formulae in Theorem \ref{mainthm} and Proposition \ref%
{prop:1}, to determine the quantities $\{|\Gamma _{n}^{p}|\}_{n\geq 0,p\geq
0},\{|\Delta _{n}^{p}|\}_{n\geq 0,p\geq 0}$, we only need to determine the
quantities for $p=0$ and $1$. The following two propositions are devoted to
this purpose.

\begin{proposition}
\label{prop:p0}With the above notation, we have
\begin{equation}
|\Gamma_{n}^{0}|\equiv \left\{
\begin{array}{cc}
1 & \text{if }n\equiv 1,2\mod 4 \\
2 & \text{if }n\equiv 3,0\mod 4%
\end{array}%
\right. ,|\Delta_{n}^{0}|\equiv \left\{
\begin{array}{cc}
2 & \text{if }n\equiv 1,2\mod 4 \\
1 & \text{if }n\equiv 3,0\mod 4%
\end{array}%
\right. .  \label{p0}
\end{equation}
\end{proposition}

\begin{proof}
We will prove these two assertion simultaneously. For $n=0,1,2,$ the above
equalities can be check directly. Assume that the equalities hold for $n\leq
k.$ By the induction hypothesis, we have for all $n<k,$%
\begin{eqnarray*}
|\Gamma _{n}^{0}|^{2}\equiv |\Delta _{n}^{0}|^{2}\equiv 1, &&|\Gamma
_{n}^{0}|\cdot |\Delta _{n}^{0}|\equiv 2, \\
|\Gamma _{n+1}^{0}|\cdot |\Delta _{n-1}^{0}|\equiv 1, &&|\Gamma
_{n+1}^{0}|+|\Delta _{n-1}^{0}|\equiv |\Delta _{n+1}^{2}|.
\end{eqnarray*}%
Then if $n=k+1=3l(l\geq 1),$ by Theorem \ref{mainthm} (1) and (10), we have%
\begin{eqnarray*}
|\Gamma _{n}^{0}| &=&(-1)^{l}|\Gamma _{l}^{0}|^{2}\cdot |\Delta
_{l}^{0}|+(-1)^{l+1}|\Gamma _{l}^{0}|\cdot |\Gamma _{l+1}^{0}|\cdot |\Delta
_{l-1}^{0}| \\
&\equiv &(-1)^{l}\left( 2|\Gamma _{l}^{0}|-|\Gamma _{l}^{0}|\right)  \\
&\equiv &(-1)^{l}|\Gamma _{l}^{0}|, \\
|\Delta _{n}^{0}| &\equiv &(-1)^{l}|\Gamma _{l}^{0}|\cdot |\Delta
_{l}^{0}|^{2}+(-1)^{l+1}|\Gamma _{l+1}^{0}|\cdot |\Delta _{l-1}^{0}|\cdot
|\Delta _{l}^{0}| \\
&\equiv &(-1)^{l}\left( 2|\Delta _{l}^{0}|-|\Delta _{l}^{0}|\right)  \\
&\equiv &(-1)^{l}|\Delta _{l}^{0}|.
\end{eqnarray*}%
Since $n=3l\equiv -l(\textrm{mod } 4)$, the above two equalities implies (\ref{p0}).

When $n=k+1=3l+1(l\geq 1),$ by formulae (2) and (11) of Theorem \ref{mainthm},
\begin{eqnarray*}
|\Gamma _{3l+1}^{0}| &=&(-1)^{l}|\Gamma _{l}^{0}|\cdot |\Gamma
_{l+1}^{0}|\cdot |\Delta _{l}^{0}|+(-1)^{l+1}|\Gamma _{l+2}^{0}|^{2}\cdot
|\Delta _{l-1}^{0}| \\
&\equiv &(-1)^{l}\left( 2|\Gamma _{l+1}^{0}|-|\Delta _{l-1}^{0}|\right)  \\
&\equiv &(-1)^{l+1}|\Delta _{l+1}^{0}|, \\
|\Delta _{3l+1}^{0}| &\equiv &(-1)^{l+1}|\Gamma _{l+1}^{0}|\cdot |\Delta
_{l}^{0}|^{2} \\
&\equiv &(-1)^{l+1}|\Gamma _{l+1}^{0}|.
\end{eqnarray*}%
Note that $n=3l+1\equiv 1-l(\textrm{mod } 4)$, (\ref{p0}) follows from the above two
equalities.

When $n=k+1=3l+2(l\geq 1),$ by formulae (3) and (12) of Theorem \ref{mainthm}, we have
\begin{eqnarray*}
|\Gamma _{3l+2}^{0}| &=&(-1)^{l}|\Gamma _{l+1}^{0}|^{2}\cdot |\Delta
_{l}^{0}| \\
&\equiv &(-1)^{l}|\Delta _{l}^{0}|, \\
|\Delta _{3l+2}^{0}| &\equiv &(-1)^{l}|\Gamma _{l+2}^{0}|\cdot |\Delta
_{l}^{0}|^{2}+(-1)^{l+1}|\Gamma _{l+1}^{0}|\cdot |\Delta _{l}^{0}|\cdot
|\Delta _{l+1}^{0}| \\
&\equiv &(-1)^{l}\left( |\Delta _{l}^{0}|-2|\Delta _{l}^{0}|\right)
=(-1)^{l+1}|\Delta _{l}^{0}|.
\end{eqnarray*}%
These two equalities, combining with the fact $n=3l+2\equiv 2-l(\textrm{mod }4),$
lead to (\ref{p0}). Thus the assertions are proved.
\end{proof}

\begin{proposition}
\label{prop:p1}For $p=1$, we have
\begin{equation}
|\Gamma _{n}^{1}|\equiv |\Delta _{n}^{1}|\equiv \left\{
\begin{array}{ll}
0 & \text{if }n\equiv 1,3\mod 4 \\
2 & \text{if }n\equiv 2\mod 4 \\
1 & \text{if }n\equiv 0\mod 4%
\end{array}%
\right. .  \label{p1}
\end{equation}
\end{proposition}

\begin{proof}
These two assertions will be proved simultaneously. For $n=0,1,2,$ the above
equalities can be check directly. Assume that the equalities hold for $n\leq
k.$ According to \ref{p0}, for all $n\geq 1,$%
\begin{equation*}
|\Gamma _{n}^{0}|^{2}\equiv |\Delta _{n}^{0}|^{2}\equiv 1,|\Gamma
_{n}^{0}|\cdot |\Delta _{n}^{0}|\equiv 2,|\Gamma _{n+1}^{0}|\cdot |\Delta
_{n-1}^{0}|\equiv 1.
\end{equation*}%
By the induction hypothesis, for all $1\leq n<k,$%
\begin{equation*}
\left\vert \Gamma _{n-1}^{1}\right\vert \equiv -\left\vert \Gamma
_{n+1}^{1}\right\vert .
\end{equation*}%
Then if $n=k+1=3l(l\geq 1),$ by formulae (4) and (13) of Theorem \ref{mainthm}, we have%
\begin{eqnarray*}
|\Gamma _{n}^{1}| &=&(-1)^{l}|\Gamma _{l}^{0}|\cdot |\Gamma _{l}^{1}|\cdot
|\Delta _{l}^{0}|+(-1)^{l+1}|\Gamma _{l}^{1}|\cdot |\Gamma _{l+1}^{0}|\cdot
|\Delta _{l-1}^{0}| \\
&\equiv &(-1)^{l}\left( 2|\Gamma _{l}^{1}|-|\Gamma _{l}^{1}|\right)  \\
&\equiv &(-1)^{l}|\Gamma _{l}^{1}|, \\
|\Delta _{n}^{1}| &\equiv &(-1)^{l}|\Gamma _{l}^{1}|\cdot |\Delta
_{l}^{0}|^{2}\equiv (-1)^{l}|\Gamma _{l}^{1}|.
\end{eqnarray*}%
Since $n=3l\equiv -l(\textrm{mod } 4)$, (\ref{p1}) holds in this case.

When $n=k+1=3l+1(l\geq 1),$ by formulae (5) and (14) and Theorem \ref{mainthm}, we have%
\begin{eqnarray*}
|\Gamma _{n}^{1}| &=&(-1)^{l+1}|\Gamma _{l+1}^{0}|^{2}\cdot |\Delta
_{l-1}^{1}|\equiv (-1)^{l+1}|\Delta _{l-1}^{1}|\equiv (-1)^{l}|\Gamma
_{l+1}^{1}|, \\
|\Delta _{n}^{1}| &\equiv &(-1)^{l}|\Gamma _{l+1}^{1}|\cdot |\Delta
_{l}^{0}|^{2}\equiv (-1)^{l}|\Gamma _{l+1}^{1}|.
\end{eqnarray*}%
Since $n=3l+1\equiv 1-l(\textrm{mod } 4)$, (\ref{p1}) holds in this case.

When $n=k+1=3l+2(l\geq 1),$ by formulae (6) and (15) of Theorem \ref{mainthm}, we have%
\begin{eqnarray*}
|\Gamma _{n}^{1}| &=&(-1)^{l+1}|\Gamma _{l+1}^{0}|^{2}\cdot |\Delta
_{l}^{1}|\equiv (-1)^{l+1}|\Delta _{l}^{1}|, \\
|\Delta _{n}^{1}| &\equiv &(-1)^{l}|\Gamma _{l+2}^{0}|\cdot |\Delta
_{l}^{0}|\cdot |\Delta _{l}^{1}|+(-1)^{l+1}|\Gamma _{l+1}^{0}|\cdot |\Delta
_{l}^{1}|\cdot |\Delta _{l+1}^{0}| \\
&\equiv &(-1)^{l}\left( |\Delta _{l}^{1}|-2|\Delta _{l}^{1}|\right) \equiv
(-1)^{l+1}|\Delta _{l}^{1}|.
\end{eqnarray*}%
These two equalities, combining with the fact $n=3l+2\equiv 2-l(\textrm{mod } 4),$
lead to (\ref{p1}). Thus the assertions are proved.
\end{proof}

\subsection{Periodicity properties.}

Let $(u_{n})_{n\geq 0}$ be a sequence with $u_{n}\in \mathbb{F}_{3}$, then
the formal power series
\begin{equation*}
u(x)=\sum_{n\geq 0}u_{n}x^{n}
\end{equation*}%
is called the \emph{generating series} of the sequence $(u_{n})_{n\geq 0}.$

A sequence $(u_{n})_{n\geq 0}$ is periodic of period $t$ if and only if
its generating series adds up to a rational fraction of the form $\frac{P(x)%
}{1-x^{t}}$, where $P(x)$ is a polynomial of degree less than $t$.

Let $P(x)=\sum_{n\geq 0}a_{n}x^{n}$ and $Q(x)=\sum_{n\geq 0}b_{n}x^{n}$ be
two formal power series with $a_{n},b_{n}\in \mathbb{F}_{3}$, then their
\emph{Hadamard product} is defined to be
\begin{equation*}
P(x)\star Q(x)=\sum_{n\geq 0}a_{n}b_{n}x^{n}.
\end{equation*}%
In addition, if $(a_{n})_{n\geq 0}$ and $(b_{n})_{n\geq 0}$ are periodic of
period $s$ and $t$ respectively, then $P(x)\star Q(x)$ is the generating
series of the periodic sequence $(a_{n}b_{n})_{n\geq 0}$ having $[s,t]$ as a
period, where $[s,t]$ denotes the lowest common multiple of $s$ and $t$.

For $p\geq 0$, define
\begin{equation*}
\begin{array}{ll}
f^{(p)}(x)=\sum_{n\geq 0}\left\vert \Gamma_{n}^{p}\right\vert x^{n}, &
g^{(p)}(x)=\sum_{n\geq 0}\left\vert D _{n}^{p}\right\vert x^{n},%
\end{array}%
\end{equation*}%
and
\begin{equation*}
h_{0}(x)=\sum_{n\geq 0}(-1)^{n}x^{n},\ h_{1}(x)=\sum_{n\geq 0}(-1)^{n+1}x^{n}
\end{equation*}%
where the coefficients are taken modulo $3$, with the convention of
Proposition \ref{prop:1}.

By Proposition \ref{prop:p0} and Proposition \ref{prop:p1}, we have
\begin{equation}
\left\{
\begin{array}{ll}
f^{(0)}=\frac{2+x+x^{2}+2x^{3}}{1-x^{4}}, & g^{(0)}=\frac{1+2x+2x^{2}+x^{3}}{%
1-x^{4}}, \\
f^{(1)}=\frac{1+2x^{2}}{1-x^{4}}, & g^{(1)}=\frac{1+2x^{2}}{1-x^{4}}.%
\end{array}%
\right.  \label{fraction}
\end{equation}

Using the recurrent formulae in Theorem \ref{mainthm}, we can compute the
above quantities recursively. We compute $f^{(2)}$ and $g^{(2)}$ as an
example.

\begin{eqnarray*}
f^{(2)} &=&\sum_{n\geq 0}|\Gamma _{n}^{2}|x^{n}=\sum_{n\geq 0}(|\Gamma
_{3n}^{2}|+|\Gamma _{3n+1}^{2}|x+|\Gamma _{3n+2}^{2}|x^{2})x^{3n} \\
&=&\sum_{n\geq 0}\Big[(-1)^{n}|\Gamma _{n}^{1}|^{2}|\Delta
_{n}^{0}|+\big((-1)^{n}|\Gamma _{n}^{1}||\Gamma _{n+1}^{0}||\Delta _{n}^{1}| \\
&&+(-1)^{n+1}|\Gamma _{n+1}^{0}||\Gamma _{n+1}^{1}||\Delta
_{n-1}^{1}|\big)x+(-1)^{n+1}|\Gamma _{n+1}^{1}|^{2}|\Delta _{n}^{0}|x^{2}\Big]x^{3n}
\\
&&\text{by Theorem \ref{mainthm}} \\
&=&(h_{0}\star f^{(1)}\star f^{(1)}\star g^{(0)})^{3}+x(h_{0}\star
f^{(1)}\star \widehat{f^{(0)}}\star g^{(1)} \\
&&+h_{1}\star \widehat{f^{(0)}}\star \widehat{f^{(1)}}\star \overline{g^{(1)}%
})^{3}+x^{2}(h_{1}\star \widehat{f^{(1)}}\star \widehat{f^{(1)}}\star
g^{(0)})^{3}
\end{eqnarray*}%
where
\begin{equation}
h_{0}(x)=\sum_{n\geq 0}(-1)^{n}x^{n},\ h_{1}(x)=\sum_{n\geq 0}(-1)^{n+1}x^{n}
\label{h0}
\end{equation}%
and
\begin{equation}
\widehat{f^{(1)}}(x)=\sum_{n\geq 0}\left\vert \Gamma _{n+1}^{1}\right\vert
x^{n},\ \overline{g^{(1)}}(x)=\sum_{n\geq 0}\left\vert \Delta
_{n-1}^{1}\right\vert x^{n}.  \label{f1}
\end{equation}%
Thus by (\ref{fraction}), (\ref{h0}) and (\ref{f1}), we have%
\begin{eqnarray*}
f^{(2)} &=&\left( \frac{1+2x^{2}}{1-x^{4}}\right) ^{3}+x\left( \frac{1+2x^{2}%
}{1-x^{4}}+\frac{2x+x^{3}}{1-x^{4}}\right) ^{3}+x^{2}\left( \frac{2x+x^{3}}{%
1-x^{4}}\right) ^{3} \\
&=&\frac{1+x+2x^{4}+2x^{5}+2x^{6}+2x^{7}+x^{10}+x^{11}}{1-x^{12}};
\end{eqnarray*}%
\begin{eqnarray*}
g^{(2)} &=&\sum_{n\geq 0}|\Delta _{n}^{2}|x^{n}=\sum_{n\geq 0}(|\Delta
_{3n}^{2}|+x|\Delta _{3n+1}^{2}|+x^{2}|\Delta _{3n+2}^{2}|)x^{3n} \\
&=&\sum_{n\geq 0}\Big[\left( (-1)^{n}|\Gamma _{n}^{1}||\Delta _{n}^{0}||\Delta
_{n}^{1}|+(-1)^{n+1}|\Gamma _{n+1}^{1}||\Delta _{n}^{0}||\Delta
_{n-1}^{1}|\right)  \\
&&+x\big((-1)^{n}|\Gamma _{n+1}^{0}||\Delta _{n}^{1}|^{2}\big)+x^{2}\big((-1)^{n}|\Gamma
_{n+2}^{0}||\Delta _{n}^{1}|^{2}\big)\Big]x^{3n} \\
&=&(h_{0}\star f^{(1)}\star g^{(0)}\star g^{(1)}+h_{1}\star \widehat{f^{(1)}}%
\star g^{(0)}\star \overline{g^{(1)}})^{3} \\
&&+x(h_{0}\star \widehat{f^{(0)}}\star g^{(1)}\star
g^{(1)})^{3}+x^{2}(h_{0}\star \widehat{\widehat{f^{(0)}}}\star g^{(1)}\star
g^{(1)})^{3}
\end{eqnarray*}%
where%
\begin{equation}
\widehat{\widehat{f^{(0)}}}(x)=\sum_{n\geq 0}\left\vert \Gamma
_{n+2}^{0}\right\vert x^{n}.  \label{f2}
\end{equation}%
Thus by (\ref{fraction}), (\ref{h0}), (\ref{f1}) and (\ref{f2}), we have%
\begin{equation*}
g^{(2)}=\frac{1+x+x^{2}+2x^{6}+2x^{7}+2x^{8}}{1-x^{12}}.
\end{equation*}

\begin{theorem}\label{thm:perodic}
For any $p\geq 0$, the sequences $($modulo $3)$%
\begin{equation*}
\{\left\vert \Gamma _{n}^{p}\right\vert \}_{n\geq 0},\{\left\vert \Delta
_{n}^{p}\right\vert \}_{n\geq 0}
\end{equation*}%
are both periodic. Moreover, $12\cdot 3^{k}$ is a period if $3^{k}+1\leq
p\leq 3^{k+1}.$
\end{theorem}

\begin{proof}
For $p=0,1,2,3$ by the recurrence formulae of Proposition \ref{prop:1} and
equalities (\ref{fraction}) these two sequences are periodic. Now suppose $%
p\geq 4$, we shall prove by induction on $k$ that $12\cdot 3^{k}$ is a
period if $3^{k}+1\leq p\leq 3^{k+1}.$

By calculation, we can find that the conclusion is true for $k=1$.
Suppose that the conclusion is true for $p\leq 3^{k}.$ We need to show that
the conclusion is true for $3^{k}+1\leq p\leq 3^{k+1}.$ If $p=3q$, then $%
3^{k-1}+1\leq q\leq 3^{k}$, thus by Theorem \ref{mainthm} (1) (2) and (3),
we have
\begin{eqnarray}
|\Gamma _{3n}^{p}| &=&(-1)^{n}|\Gamma _{n}^{q}|^{2}\cdot |\Delta
_{n}^{q}|+(-1)^{n+1}|\Gamma _{n}^{q}|\cdot |\Gamma _{n+1}^{q}|\cdot |\Delta
_{n-1}^{q}|,  \notag \\
|\Gamma _{3n+1}^{p}| &=&(-1)^{n}|\Gamma _{n}^{q}|\cdot |\Gamma
_{n+1}^{q}|\cdot |\Delta _{n}^{q}|+(-1)^{n+1}|\Gamma _{n+2}^{q}|^{2}\cdot
|\Delta _{n-1}^{q}|,  \label{inproof1} \\
|\Gamma _{3n+2}^{p}| &=&(-1)^{n}|\Gamma _{n+1}^{q}|^{2}\cdot |\Delta
_{n}^{q}|.  \notag
\end{eqnarray}%
By the induction hypothesis, all the sequences occurring on the right hand
of (\ref{inproof1}) have period $12\cdot 3^{k-1}$, and so do the product and
sum of these sequences. Therefore, the sequences $\left\vert \Gamma
_{3n}^{p}\right\vert ,\left\vert \Gamma _{3n+1}^{p}\right\vert $ and $%
\left\vert \Gamma _{3n+2}^{p}\right\vert $ are all $12\cdot 3^{k-1}$%
-periodic which implies that the sequence $\left\vert \Gamma
_{n}^{p}\right\vert $ is of period $12\cdot 3^{k}$. The cases $p=3q+1$ and $%
3q+2$ follow in the same way. Similar discussions can be applied to the
sequence $\left\vert \Delta _{n}^{p}\right\vert $.
\end{proof}

\subsection{Automaticity properties.}
First, we will recall some definitions of two dimensional automatic sequences which can be found in \cite[Chapter 14]{ALL}.

Let $\mathcal{A}, \mathcal{B}$ be two finite alphabets. If $$A=(a_{i,j})_{0\leq i\leq m, 0\leq j\leq n}$$ is an $m\times n$ matrix with entries in $\mathcal{A}$, and $\psi: \mathcal{A}\rightarrow \mathcal{B}^{k\times l}$ is a \emph{$[k,l]$-uniform matrix-valued morphism}, i.e., a map sending each letter in $\mathcal{A}$ to an $k\times l$ matrix, then $\psi(A)$ is an $km\times ln$ matrix given by
\begin{equation*}
\left[
\begin{array}{cccc}
\psi(a_{00}) & \psi(a_{01}) & \cdots & \psi(a_{0,n-1})\\
\psi(a_{10}) & \psi(a_{11}) & \cdots & \psi(a_{1,n-1})\\
\vdots & \vdots & \ddots &\vdots\\
\psi(a_{m-1,0}) & \psi(a_{m-1,1}) & \cdots & \psi(a_{m-1,n-1})\\
\end{array}
\right].
\end{equation*}

A \emph{$[k,l]$-automatic sequence} is the image (under a coding) of a fixed point of a $[k,l]$-morphism. In particular, if $k=l$, the $[k,k]$-automatic sequence is also called the $k$-\emph{automatic sequence}. A well known result \cite[Theorem 14.2.2]{ALL}(see also \cite{Sal},\cite{Sal2}) shows that the two-dimensional sequence $\mathbf{u}=(u_{n,m})_{n,m\geq 0}$ is $[k,l]$-automatic sequences if and only if the $[k,l]$-kernel $K_{k,l}(\mathbf{u})$ is finite, where
\[K_{k,l}(\mathbf{u})=\{(u_{k^am+r,l^an+s})_{m,n\geq 0}:a\geq 0, 0\leq r <k^a, 0\leq s <l^a\}.\]

\begin{theorem}
\label{thm:auto}The two-dimensional sequences $($modulo $3)$%
\begin{equation*}
\{\left\vert \Gamma _{n}^{p}\right\vert \}_{n\geq 0,p\geq 0},\{\left\vert
\Delta _{n}^{p}\right\vert \}_{n\geq 0,p\geq 0}
\end{equation*}%
are both $3$-automatic.
\end{theorem}

\begin{proof}
For this, we only need to show that the $3$-kernels of these sequences are
finite.

Let $\{u_{n}^{p}\}_{n\geq -1,p\geq 0}$ be a double sequence. For $\alpha^{\prime} \in
\{-1,0,1,2\}$ and $\alpha,\beta,\beta^{\prime} \in \{0,1,2\}$, operations $S_{\alpha^{\prime} }^{\beta^{\prime} }$
and $T_{\alpha }^{\beta }$ are defined as follows%
\begin{equation*}
S_{\alpha^{\prime} }^{\beta^{\prime} }u=\{u_{n+\alpha^{\prime} }^{p+\beta^{\prime} }\}_{n\geq 0,p\geq 0},\
T_{\alpha }^{\beta }u=\{u_{3n+\alpha }^{3p+\beta }\}_{n\geq 0,p\geq 0}.
\end{equation*}%
Then, for $\alpha ^{\prime }\in \{-1,0,1,2\}$ and $\alpha ,\beta ,\beta
^{\prime }\in \{0,1,2\}$, we have%
\begin{equation}
T_{\alpha}^{\beta}S_{\alpha^{\prime}}^{\beta^{\prime}}=\left\{
	\begin{array}{ll}
	S_{-1}^0T_{2}^{\beta+\beta^{\prime}} & \text{ if } \alpha+\alpha^{\prime}=-1 \text{ and } \beta+\beta^{\prime}\leq 2,\\
	S_{-1}^1T_{2}^{\beta+\beta^{\prime}-3} & \text{ if } \alpha+\alpha^{\prime}=-1 \text{ and } \beta+\beta^{\prime}\geq 3,\\
	T_{\alpha+\alpha^{\prime}}^{\beta+\beta^{\prime}} & \text{ if } 0\leq\alpha+\alpha^{\prime}\leq 2 \text{ and } \beta+\beta^{\prime}\leq 2,\\
	S_{0}^1T_{\alpha+\alpha^{\prime}}^{\beta+\beta^{\prime}-3} & \text{ if } 0\leq\alpha+\alpha^{\prime}\leq 2 \text{ and } \beta+\beta^{\prime}\geq 3,\\
	S_{1}^0T_{\alpha+\alpha^{\prime}-3}^{\beta+\beta^{\prime}} & \text{ if } \alpha+\alpha^{\prime}\geq 3 \text{ and } \beta+\beta^{\prime}\leq 2,\\
	S_{1}^1T_{\alpha+\alpha^{\prime}-3}^{\beta+\beta^{\prime}-3} & \text{ if } \alpha+\alpha^{\prime}\geq 3 \text{ and } \beta+\beta^{\prime}\geq 3.
	\end{array}
	\right.  \label{inproof2}
\end{equation}%

Suppose $\Gamma,\Delta$ and $F$ stand for the sequences $\{\left\vert \Gamma
_{n}^{p}\right\vert \}_{n\geq 0,p\geq 0},\{\left\vert \Delta
_{n}^{p}\right\vert \}_{n\geq 0,p\geq 0}$ and $\{F_{n}^{p}\}_{n\geq 0,p\geq
0}$ modulo $3$ where $F_{n}^{p}=(-1)^{n}.$ Thus for any $\beta \in \{0,1,2\}$%
\begin{equation}
T_{0}^{\beta }F=T_{2}^{\beta }F=S_{2}^{\beta }F=F\text{ and }T_{1}^{\beta
}F=S_{1}^{\beta }F=S_{1}^{0}F.  \label{inproof22}
\end{equation}

We rewrite Theorem \ref{mainthm} and Proposition \ref{prop:1} as follows%
\begin{equation}
\left\{
\begin{array}{ll}
T_{0}^{0}\Gamma\equiv F\cdot \Gamma^{2}\cdot \Delta  & T_{0}^{2}\Gamma\equiv F\cdot \left(
S_{0}^{1}\Gamma\right) ^{2}\cdot \Delta , \\
\multicolumn{1}{r}{+S_{1}^{0}F\cdot \Gamma\cdot S_{1}^{0}\Gamma\cdot S_{-1}^{0}\Delta ,
} & \multicolumn{1}{r}{} \\
T_{0}^{0}\Delta \equiv F\cdot \Gamma\cdot \Delta ^{2} & T_{0}^{2}\Delta \equiv
F\cdot S_{0}^{1}\Gamma\cdot \Delta \cdot S_{0}^{1}\Delta  \\
\multicolumn{1}{r}{+S_{1}^{0}F\cdot S_{1}^{0}\Gamma\cdot S_{-1}^{0}\Delta \cdot
\Delta ,} & \multicolumn{1}{r}{+S_{1}^{0}F\cdot S_{1}^{1}\Gamma\cdot \Delta \cdot
S_{-1}^{1}\Delta ,} \\
T_{1}^{0}\Gamma\equiv F\cdot S_{1}^{0}\Gamma\cdot \Gamma\cdot \Delta  & T_{1}^{2}\Gamma\equiv
F\cdot S_{1}^{0}\Gamma\cdot S_{0}^{1}\Gamma\cdot S_{0}^{1}\Delta  \\
\multicolumn{1}{r}{+S_{1}^{0}F\cdot (S_{1}^{0}\Gamma)^{2}\cdot S_{-1}^{0}\Delta ,}
& \multicolumn{1}{r}{+S_{1}^{0}F\cdot S_{1}^{0}\Gamma\cdot S_{1}^{1}\Gamma\cdot
S_{-1}^{1}\Delta ,} \\
T_{1}^{0}\Delta \equiv S_{1}^{0}F\cdot S_{1}^{0}\Gamma\cdot \Delta ^{2}, &
T_{1}^{2}\Delta \equiv F\cdot S_{1}^{0}\Gamma\cdot \left( S_{0}^{1}\Delta \right)
^{2}, \\
T_{2}^{0}\Gamma\equiv F\cdot (S_{1}^{0}\Gamma)^{2}\cdot \Delta , & T_{2}^{2}\Gamma\equiv
S_{1}^{0}F\cdot \left( S_{1}^{1}\Gamma\right) ^{2}\cdot \Delta , \\
T_{2}^{0}\Delta \equiv F\cdot S_{2}^{0}\Gamma\cdot \Delta ^{2} & T_{2}^{2}\Delta
\equiv F\cdot S_{2}^{0}\Gamma\cdot \left( S_{0}^{1}\Delta \right) ^{2}, \\
\multicolumn{1}{r}{+S_{1}^{0}F\cdot S_{1}^{0}\Gamma\cdot \Delta \cdot
S_{1}^{0}\Delta ,} & \multicolumn{1}{r}{} \\
T_{0}^{1}\Gamma\equiv F\cdot S_{0}^{1}\Gamma\cdot \Gamma\cdot \Delta  & T_{2}^{1}\Gamma\equiv
S_{1}^{0}F\cdot \left( S_{1}^{0}\Gamma\right) ^{2}\cdot S_{0}^{1}\Delta , \\
\multicolumn{1}{r}{+S_{1}^{0}F\cdot S_{1}^{0}\Gamma\cdot S_{0}^{1}\Gamma\cdot
S_{-1}^{0}\Delta ,} & \multicolumn{1}{r}{} \\
T_{0}^{1}\Delta \equiv F\cdot S_{0}^{1}\Gamma\cdot \Delta ^{2}, & T_{2}^{1}\Delta
\equiv F\cdot S_{2}^{0}\Gamma\cdot \Delta \cdot S_{0}^{1}\Delta  \\
\multicolumn{1}{r}{} & \multicolumn{1}{r}{+S_{1}^{0}F\cdot S_{1}^{0}\Gamma\cdot
S_{0}^{1}\Delta \cdot S_{1}^{0}\Delta ,} \\
T_{1}^{1}\Gamma\equiv S_{1}^{0}F\cdot \left( S_{1}^{0}\Gamma\right) ^{2}\cdot
S_{-1}^{1}\Delta , & T_{1}^{1}\Delta \equiv F\cdot S_{1}^{1}\Gamma\cdot \Delta
^{2}.%
\end{array}%
\right.   \label{inproof3}
\end{equation}

Let $\mathcal{X}=\{\Gamma,\Delta,F\}$ and $\mathcal{Y}=\{S_{\alpha }^{\beta
}J~| ~J\in \mathcal{X},\alpha=-1,0,1,2 \text{ and }\beta =0,1,2\}.$ According to (\ref{inproof2}) (\ref%
{inproof22}) and (\ref{inproof3}), for any $\alpha ,\beta \in \{0,1,2\}$ and
$J\in \mathcal{Y}$, $T_{\alpha }^{\beta }J$ can be expressed as a polynomial
with coefficients in $GF_{3}$ of the elements of $\mathcal{Y}$. Hence the
elements of $3$-kernels $K_3(\Gamma)$ and $K_3(\Delta)$ are obtained by successive
applications of operators $T_{\alpha }^{\beta }$. For instance, let $(|\Gamma_{3^mn+r}^{3^mp+s}|)_{n,p\geq 0}\in K_3(\Gamma)$ where $m\geq 0$ and $0\geq r,s\geq 3^m$. Suppose $r=\sum_{i=0}^{m-1}3^i\alpha_i$ and $s=\sum_{i=0}^{m-1}3^i\beta_i$ where $\alpha_i,\beta_i\in \{0,1,2\}$. It is easy to verify that  $(|\Gamma_{3^mn+r}^{3^mp+s}|)_{n,p\geq 0}=T_{\alpha_{m-1}}^{\beta_{m-1}}\cdots T_{\alpha_{1}}^{\beta_{1}}T_{\alpha_{0}}^{\beta_{0}}(\Gamma)$. 
Therefore the elements of these two $3$-kernels are polynomials with
coefficients in $GF_{3}$ of the elements of $\mathcal{Y}$.

Since there is only a finite number of polynomial functions on $GF_{3}$ with
twelve variables, these $3$-kernels are finite. Therefore, the
sequences $\{\left\vert \Gamma _{n}^{p}\right\vert \}_{n\geq 0,p\geq 0}$ and
$\{\left\vert \Delta _{n}^{p}\right\vert \}_{n\geq 0,p\geq 0}$ are $3$%
-automatic.
\end{proof}

\begin{corollary}
For any $n\geq 1$, the sequences $($modulo $3)$%
\begin{equation*}
\{\left\vert \Gamma _{n}^{p}\right\vert \}_{p\geq 0},\{\left\vert \Delta
_{n}^{p}\right\vert \}_{p\geq 0}
\end{equation*}%
are both $3$-automatic.
\end{corollary}

\begin{proof}
An immediate consequence of Salon \cite{Sal}, \cite{Sal2} is that, if a
two-dimensional sequence $\{s_{m,n}\}_{m,n\geq 0}$ is 3-automatic, then for
any fixed $m\geq 0$ the sequence $\{s_{m,n}\}_{n\geq 0}$ is 3-automatic
which prove our result.
\end{proof}

\section{Applications.}
\subsection{Pad\'{e} approximation}
Now, consider again the Cantor sequence
$$\boldsymbol{c}=c_{0}c_{1}c_{2}\cdots \in \{0,1\}^{\mathbb{N}},$$
let
\begin{equation*}
f(x)=\sum_{n\geq 0}c_{n}x^{n}
\end{equation*}%
be the generating function of the Cantor sequence. It follows from (\ref%
{cantor}) that
\begin{equation}
f(x)=(1+x^{2})f(x^{3}),  \label{eqn:f}
\end{equation}%
and $f(x)>1$ for any $x>0$.

Denote by $\left[ \frac{p}{q}\right] _{f}$, a $(p,q)$-order Pad\'{e}
approximate of $f$, i.e., a rational function $P(x)/Q(x)$ whose denominator
has degree $q$ and whose numerator has degree $p$ such that
\begin{equation*}
f(x)-\frac{P(x)}{Q(x)}=O(x^{p+q+1}),\ x\longrightarrow 0.
\end{equation*}%
A classical result \cite[Brezinski, Page 35]{Bre}, shows us that if $\Gamma_{n}^{0}\neq 0$, then the Pad\'{e}
approximate $\left[ \frac{n-1}{n}\right] _{f}$ exists. Moreover,
\begin{equation}
f(x)-\left[ n-1/n\right] _{f}(x)=\frac{\Gamma_{n+1}^{0}}{\Gamma_{n}^{0}}%
x^{2n}+O(x^{2n+1}).  \label{eqn:pade}
\end{equation}%
Hence by Proposition \ref{prop:p0}, we have the following theorem.

\begin{theorem}
\label{thm:pade}Let $f(x)=\sum_{n\geq 0}c_{n}x^{n}$, then for any $n\geq 1$,
the $(n-1,n)$-order Pad\'{e} approximate of $f$ exists.
\end{theorem}

\subsection{The irrationality exponent of the Cantor number}
Let $\xi$ be an irrational number. The \emph{irrationality exponent} (or \emph{irrational measure})
$\mu(\xi)$ of $\xi$ is defined as follow
\[\mu(\xi)=\sup\left\{\mu\in\mathbb{R}: \left|\xi-\frac{p}{q}\right|<\frac{1}{q^{\mu}} \text{ for infinite many } (p,q)\in \mathbb{Z}\times\mathbb{N}\right\}.\]

\medskip
Let $\xi_{\mathbf{c},b}$ be the \emph{Cantor number} defined by
\[\xi_{\mathbf{c},b}=\sum_{k\geq 0}\frac{c_k}{b^k}=1+\frac{1}{b^2}+\frac{1}{b^6}+\frac{1}{b^8}+\cdots,\]
where $\mathbf{c}=c_0c_1c_2\cdots$ is the Cantor sequence. Combine equation (\ref{eqn:f}), Propsition \ref{prop:1} and Theorem 1 in \cite{GWW}, we have
\begin{proposition}
\label{irrationality}For any integer $b\geq 2$, the irrationality exponent
of the Cantor number $\xi _{\boldsymbol{c},b}$ is equal to $2$.
\end{proposition}

\begin{corollary}
For any $b\geq 2,$ let $\eta _{\mathbf{d },b}=\sum_{n\geq 0}d _{n}b^{-n}$.
The irrationality exponent of $\eta _{\mathbf{d },b}$ is equal to $2$.
\end{corollary}

\begin{proof}
Using (\ref{delta}), we have
\begin{eqnarray*}
\eta _{\mathbf{d },b} &=&\sum_{n\geq 0}d _{n}b^{-n}=\sum_{n\geq 0}(d _{3n}+d
_{3n+1}b^{-1}+d _{3n+2}b^{-2})b^{-3n} \\
&=&\sum_{n\geq 0}(2c_{n}+c_{n+1}b^{-1}+c_{n}b^{-2})b^{-3n} \\
&=&2\xi _{\boldsymbol{c},b^{3}}+b^{-1}(\xi _{\boldsymbol{c}%
,b^{3}}-1)+b^{-2}\xi _{\boldsymbol{c},b^{3}} \\
&=&(2+b^{-1}+b^{-2})\xi _{\boldsymbol{c},b^{3}}-b^{-1}.
\end{eqnarray*}%
Since the irrationality exponent is invariant under multiplication and
addition of a rational number, we can deduce from Proposition \ref{irrationality}
that the irrationality exponent of $\eta _{\mathbf{d },b}$ is equal to $2.$
\end{proof}

\appendix
\section{Proof of Theorem \ref{mainthm} : Continue.}
\begin{proof}
$4)$ Combine (\ref{kc}) and (\ref{h3n}), we have
\begin{eqnarray*}
|P^{t}\Gamma _{3n}^{3p+1}P| &=&\left\vert
\begin{matrix}
\mathbf{0}_{n\times n} & \Gamma _{n}^{p} & \Gamma _{n}^{p+1} \\
\Gamma _{n}^{p} & \Gamma _{n}^{p+1} & \mathbf{0}_{n\times n} \\
\Gamma _{n}^{p+1} & \mathbf{0}_{n\times n} & \Gamma _{n}^{p+1}%
\end{matrix}%
\right\vert  \\
&=&\left\vert
\begin{matrix}
-\Gamma _{n}^{p+1} & \Gamma _{n}^{p} & \mathbf{0}_{n\times n} \\
\Gamma _{n}^{p} & \Gamma _{n}^{p+1} & \mathbf{0}_{n\times n} \\
\mathbf{0}_{n\times n} & \mathbf{0}_{n\times n} & \Gamma _{n}^{p+1}%
\end{matrix}%
\right\vert .
\end{eqnarray*}%
Since
\begin{equation*}
\left\vert
\begin{matrix}
-\Gamma _{n}^{p+1} & \Gamma _{n}^{p} \\
\Gamma _{n}^{p} & \Gamma _{n}^{p+1}%
\end{matrix}%
\right\vert =\left\vert
\begin{matrix}
\Gamma _{n}^{p} & \Gamma _{n}^{p+1} \\
\Gamma _{n}^{p+1} & -\Gamma _{n}^{p}%
\end{matrix}%
\right\vert ,
\end{equation*}%
by Lemma \ref{lem:matirx},
\begin{eqnarray*}
|\Gamma _{3n}^{3p+1}| &=&|\Gamma _{n}^{p+1}|\cdot \left\vert
\begin{matrix}
\Gamma _{n}^{p} & \Gamma _{n}^{p+1} \\
\Gamma _{n}^{p+1} & -\Gamma _{n}^{p}%
\end{matrix}%
\right\vert  \\
&=&(-1)^{n}|\Gamma _{n}^{p+1}|\cdot |\Gamma _{n}^{p}|\cdot |\Delta
_{n}^{p}|+(-1)^{n+1}|\Gamma _{n}^{p+1}|\cdot |\Gamma _{n+1}^{p}|\cdot
|\Delta _{n-1}^{p}|.
\end{eqnarray*}

$5)$ Combine (\ref{kc}) and (\ref{h3n+1}), we have
\begin{equation*}
|P^{t}\Gamma _{3n+1}^{3p+1}P|=\left\vert
\begin{array}{ccc}
\mathbf{0}_{(n+1)\times (n+1)} & (\Gamma _{n+1}^{p})^{(n+1)} & (\Gamma
_{n+1}^{p+1})^{(n+1)} \\
(\Gamma _{n+1}^{p})_{(n+1)} & \Gamma _{n}^{p+1} & \mathbf{0}_{n\times n} \\
(\Gamma _{n+1}^{p+1})_{(n+1)} & \mathbf{0}_{n\times n} & \Gamma _{n}^{p+1}%
\end{array}%
\right\vert .
\end{equation*}%
Recall that $\alpha _{p}^{n}$ is the column vector of the form $%
(c_{p},c_{p+1},\cdots ,c_{p+n-1})^{t}$, then
\begin{eqnarray*}
|P^{t}\Gamma _{3n+1}^{3p+1}P| &=&\left\vert
\begin{array}{ccc}
\mathbf{0}_{(n+1)\times (n+1)} & (\Gamma _{n+1}^{p})^{(n+1)} & (\Gamma
_{n+1}^{p+1})^{(n+1)} \\
(\Gamma _{n+1}^{p})_{(n+1)} & \Gamma _{n}^{p+1} & \mathbf{0}_{n\times n} \\
\mathbf{0}_{(n-1)\times (n+1)} & -(\Gamma _{n}^{p+1})_{(1)} & (\Gamma
_{n}^{p+1})_{(n)} \\
(\alpha _{p+n}^{n+1})^{t} & \mathbf{0}_{1\times n} & (\alpha _{p+n}^{n})^{t}%
\end{array}%
\right\vert  \\
&=&\left\vert
\begin{array}{c}
(\Gamma _{n+1}^{p})_{(n+1)} \\
(\alpha _{p+n}^{n+1})^{t}%
\end{array}%
\right\vert \cdot \left\vert
\begin{array}{cc}
(\Gamma _{n+1}^{p})^{(n+1)} & (\Gamma _{n+1}^{p+1})^{(n+1)} \\
-(\Gamma _{n}^{p+1})_{(1)} & (\Gamma _{n}^{p+1})_{(n)}%
\end{array}%
\right\vert  \\
&=&\left\vert \Gamma _{n+1}^{p}\right\vert \cdot \left\vert
\begin{array}{cc}
(\Gamma _{n+1}^{p})^{(n+1)} & (\Gamma _{n+1}^{p+1})^{(n+1)} \\
-(\Gamma _{n}^{p+1})_{(1)} & (\Gamma _{n}^{p+1})_{(n)}%
\end{array}%
\right\vert  \\
&=&\left\vert \Gamma _{n+1}^{p}\right\vert \cdot \left\vert
\begin{array}{ccc}
(\Gamma _{n+1}^{p})^{(n+1)} & 0_{(n+1)\times (n-1)} & \alpha _{p+n}^{n+1} \\
-(\Gamma _{n}^{p+1})_{(1)} & \Gamma _{n-1}^{p+1}+\Gamma _{n-1}^{p+3} &
\alpha _{p+n}^{n-1}%
\end{array}%
\right\vert  \\
&=&\left\vert \Gamma _{n+1}^{p}\right\vert \cdot (-1)^{n+1}\cdot \left\vert
\begin{array}{cc}
(\Gamma _{n+1}^{p})^{(n+1)} & \alpha _{p+n}^{n+1}%
\end{array}%
\right\vert \cdot \left\vert \Gamma _{n-1}^{p+1}+\Gamma
_{n-1}^{p+3}\right\vert  \\
&=&(-1)^{n+1}\left\vert \Gamma _{n+1}^{p}\right\vert ^{2}\cdot |\Gamma
_{n-1}^{p+1}+\Gamma _{n-1}^{p+3}|.
\end{eqnarray*}%
Therefore,
\begin{equation*}
|\Gamma _{3n+1}^{3p+1}|=(-1)^{n+1}|\Gamma _{n+1}^{p}|^{2}\cdot |\Delta
_{n-1}^{p+1}|.
\end{equation*}

$6)$ Combine (\ref{kc}) and (\ref{h3n+2}), we have
\begin{eqnarray*}
|P^{t}\Gamma _{3n+2}^{3p+1}P| &=&\left\vert
\begin{matrix}
\mathbf{0}_{(n+1)\times (n+1)} & \Gamma _{n+1}^{p} & (\Gamma
_{n+1}^{p+1})^{(n+1)} \\
\Gamma _{n+1}^{p} & \Gamma _{n+1}^{p+1} & \mathbf{0}_{(n+1)\times n} \\
(\Gamma _{n+1}^{p+1})_{(n+1)} & \mathbf{0}_{n\times (n+1)} & \Gamma
_{n}^{p+1}%
\end{matrix}%
\right\vert  \\
&=&\left\vert
\begin{matrix}
\mathbf{0}_{(n+1)\times (n+1)} & \Gamma _{n+1}^{p} & \mathbf{0}_{(n+1)\times
n} \\
\Gamma _{n+1}^{p} & \Gamma _{n+1}^{p+1} & -(\Gamma _{n+1}^{p+1})^{(1)} \\
(\Gamma _{n+1}^{p+1})_{(n+1)} & \mathbf{0}_{n\times (n+1)} & \Gamma
_{n}^{p+1}%
\end{matrix}%
\right\vert  \\
&=&(-1)^{n+1}\left\vert \Gamma _{n+1}^{p}\right\vert \cdot \left\vert
\begin{matrix}
\Gamma _{n+1}^{p} & -(\Gamma _{n+1}^{p+1})^{(1)} \\
(\Gamma _{n+1}^{p+1})_{(n+1)} & \Gamma _{n}^{p+1}%
\end{matrix}%
\right\vert  \\
&=&(-1)^{n+1}\left\vert \Gamma _{n+1}^{p}\right\vert \cdot \left\vert
\begin{matrix}
\Gamma _{n+1}^{p} & -(\Gamma _{n+1}^{p+1})^{(1)} \\
\mathbf{0}_{n\times (n+1)} & \Gamma _{n}^{p+1}+\Gamma _{n}^{p+3}%
\end{matrix}%
\right\vert  \\
&=&(-1)^{n+1}\left\vert \Gamma _{n+1}^{p}\right\vert ^{2}\cdot \left\vert
\Delta _{n}^{p+1}\right\vert .
\end{eqnarray*}%
Therefore,
\begin{equation*}
|\Gamma _{3n++2}^{3p+1}|=(-1)^{n+1}\left\vert \Gamma _{n+1}^{p}\right\vert
^{2}\cdot \left\vert \Delta _{n}^{p+1}\right\vert .
\end{equation*}

$7)$ Combine (\ref{kc}) and (\ref{h3n}), we have
\begin{eqnarray*}
|P^{t}\Gamma _{3n}^{3p+2}P| &=&\left\vert
\begin{matrix}
\Gamma _{n}^{p} & \Gamma _{n}^{p+1} & \mathbf{0}_{n\times n} \\
\Gamma _{n}^{p+1} & \mathbf{0}_{n\times n} & \Gamma _{n}^{p+1} \\
\mathbf{0}_{n\times n} & \Gamma _{n}^{p+1} & \Gamma _{n}^{p+2}%
\end{matrix}%
\right\vert  \\
&=&\left\vert
\begin{matrix}
\Gamma _{n}^{p} & \Gamma _{n}^{p+1} & \mathbf{0}_{n\times n} \\
\mathbf{0}_{n\times n} & \mathbf{0}_{n\times n} & \Gamma _{n}^{p+1} \\
-\Gamma _{n}^{p+2} & \Gamma _{n}^{p+1} & \Gamma _{n}^{p+2}%
\end{matrix}%
\right\vert  \\
&=&\left\vert
\begin{matrix}
\Gamma _{n}^{p} & \Gamma _{n}^{p+1} & \mathbf{0}_{n\times n} \\
\mathbf{0}_{n\times n} & \mathbf{0}_{n\times n} & \Gamma _{n}^{p+1} \\
-\Gamma _{n}^{p+2}-\Gamma _{n}^{p} & \mathbf{0}_{n\times n} & \Gamma
_{n}^{p+2}%
\end{matrix}%
\right\vert  \\
&=&(-1)^{n}|\Gamma _{n}^{p+1}|^{2}\cdot |\Gamma _{n}^{p}+\Gamma _{n}^{p+2}|.
\end{eqnarray*}%
Therefore,
\begin{equation*}
|\Gamma _{3n}^{3p+2}|=(-1)^{n}|\Gamma _{n}^{p+1}|^{2}\cdot |\Delta _{n}^{p}|.
\end{equation*}

$8)$ Combine (\ref{kc}) and (\ref{h3n+1}), we have
\begin{eqnarray*}
|P^{t}\Gamma _{3n+1}^{3p+2}P| &=&\left\vert
\begin{array}{ccc}
\Gamma _{n+1}^{p} & (\Gamma _{n+1}^{p+1})^{(n+1)} & \mathbf{0}_{(n+1)\times
n} \\
(\Gamma _{n+1}^{p+1})_{(n+1)} & \mathbf{0}_{n\times n} & \Gamma _{n}^{p+1}
\\
\mathbf{0}_{n\times (n+1)} & \Gamma _{n}^{p+1} & \Gamma _{n}^{p+2}%
\end{array}%
\right\vert  \\
&=&\left\vert
\begin{array}{ccc}
\Gamma _{n+1}^{p} & \mathbf{0}_{(n+1)\times n} & \mathbf{0}_{(n+1)\times n}
\\
(\Gamma _{n+1}^{p+1})_{(n+1)} & -\Gamma _{n}^{p+2} & \Gamma _{n}^{p+1} \\
\mathbf{0}_{n\times (n+1)} & \Gamma _{n}^{p+1} & \Gamma _{n}^{p+2}%
\end{array}%
\right\vert  \\
&=&|\Gamma _{n+1}^{p}|\cdot \left\vert
\begin{array}{cc}
-\Gamma _{n}^{p+2} & \Gamma _{n}^{p+1} \\
\Gamma _{n}^{p+1} & \Gamma _{n}^{p+2}%
\end{array}%
\right\vert  \\
&=&|\Gamma _{n+1}^{p}|\cdot \left\vert
\begin{array}{c}
\Gamma _{n}^{p+1}\Gamma _{n}^{p+2} \\
\Gamma _{n}^{p+2}-\Gamma _{n}^{p+1}%
\end{array}%
\right\vert .
\end{eqnarray*}%
By Lemma \ref{lem:matirx},
\begin{equation*}
|\Gamma _{3n+1}^{3p+2}|=(-1)^{n}|\Gamma _{n+1}^{p}|\cdot |\Gamma
_{n}^{p+1}|\cdot |\Delta _{n}^{p+1}|+(-1)^{n+1}|\Gamma _{n+1}^{p}|\cdot
|\Gamma _{n+1}^{p+1}|\cdot |\Delta _{n-1}^{p+1}|.
\end{equation*}

$9)$ Combine (\ref{kc}) and (\ref{h3n+2}), we have
\begin{eqnarray*}
|P^{t}\Gamma _{3n+2}^{3p+2}P| &=&\left\vert
\begin{array}{ccc}
\Gamma _{n+1}^{p} & \Gamma _{n+1}^{p+1} & \mathbf{0}_{(n+1)\times n} \\
\Gamma _{n+1}^{p+1} & \mathbf{0}_{(n+1)\times (n+1)} & (\Gamma
_{n+1}^{p+1})^{(n+1)} \\
\mathbf{0}_{n\times (n+1)} & (\Gamma _{n+1}^{p+1})_{(n+1)} & \Gamma
_{n}^{p+2}%
\end{array}%
\right\vert  \\
&=&\left\vert
\begin{array}{ccc}
\Gamma _{n+1}^{p} & \Gamma _{n+1}^{p+1} & -(\Gamma _{n+1}^{p})^{(n+1)} \\
(\Gamma _{n+1}^{p+1})_{(n+1)} & \mathbf{0}_{(n+1)\times (n+1)} & \mathbf{0}%
_{(n+1)\times n} \\
\mathbf{0}_{n\times (n+1)} & \Gamma _{n}^{p+1} & \Gamma _{n}^{p+2}%
\end{array}%
\right\vert  \\
&=&\left\vert
\begin{array}{ccc}
\Gamma _{n+1}^{p} & \Gamma _{n+1}^{p+1} & -(\Gamma _{n+1}^{p})^{(n+1)} \\
(\Gamma _{n+1}^{p+1})_{(n+1)} & \mathbf{0}_{(n+1)\times (n+1)} & \mathbf{0}%
_{(n+1)\times n} \\
-(\Gamma _{n+1}^{p})_{(n+1)} & \mathbf{0}_{n\times (n+1)} & \Gamma
_{n}^{p}+\Gamma _{n}^{p+2}%
\end{array}%
\right\vert  \\
&=&(-1)^{n+1}|\Gamma _{n+1}^{p+1}|\cdot \left\vert
\begin{array}{cc}
\Gamma _{n+1}^{p+1} & -(\Gamma _{n+1}^{p})^{(n+1)} \\
\mathbf{0}_{n\times (n+1)} & \Gamma _{n}^{p}+\Gamma _{n}^{p+2}%
\end{array}%
\right\vert  \\
&=&(-1)^{n+1}|\Gamma _{n+1}^{p+1}|^{2}\cdot |\Gamma _{n}^{p}+\Gamma
_{n}^{p+2}|.
\end{eqnarray*}%
Therefore,
\begin{equation*}
|\Gamma _{3n+2}^{3p+2}|=(-1)^{n+1}|\Gamma _{n+1}^{p+1}|^{2}\cdot |\Delta
_{n}^{p}|.
\end{equation*}

$10)$ Combine (\ref{kd}) and (\ref{h3n}), we have
\begin{eqnarray*}
|P^{t}\Delta _{3n}^{3p}P| &=&\left\vert
\begin{array}{ccc}
2\Gamma _{n}^{p} & \Gamma _{n}^{p+1} & \Gamma _{n}^{p} \\
\Gamma _{n}^{p+1} & \Gamma _{n}^{p} & 2\Gamma _{n}^{p+1} \\
\Gamma _{n}^{p} & 2\Gamma _{n}^{p+1} & \Gamma _{n}^{p+2}%
\end{array}%
\right\vert  \\
&\equiv &\left\vert
\begin{array}{ccc}
-\Gamma _{n}^{p} & \Gamma _{n}^{p+1} & \Gamma _{n}^{p} \\
\Gamma _{n}^{p+1} & \Gamma _{n}^{p} & -\Gamma _{n}^{p+1} \\
\Gamma _{n}^{p} & -\Gamma _{n}^{p+1} & \Gamma _{n}^{p+2}%
\end{array}%
\right\vert  \\
&=&\left\vert
\begin{array}{ccc}
-\Gamma _{n}^{p} & \Gamma _{n}^{p+1} & \boldsymbol{0}_{n\times n} \\
\Gamma _{n}^{p+1} & \Gamma _{n}^{p} & \boldsymbol{0}_{n\times n} \\
\Gamma _{n}^{p} & -\Gamma _{n}^{p+1} & \Gamma _{n}^{p}+\Gamma _{n}^{p+2}%
\end{array}%
\right\vert .
\end{eqnarray*}%
Hence, by Lemma \ref{lem:matirx},
\begin{eqnarray*}
|\Delta _{3n}^{3p}| &\equiv &|\Gamma _{n}^{p}+\Gamma _{n}^{p+2}|\cdot
\left\vert
\begin{array}{cc}
-\Gamma _{n}^{p} & \Gamma _{n}^{p+1} \\
\Gamma _{n}^{p+1} & \Gamma _{n}^{p}%
\end{array}%
\right\vert  \\
&=&(-1)^{n}|\Gamma _{n}^{p}|\cdot |\Delta _{n}^{p}|^{2}+(-1)^{n+1}|\Gamma
_{n+1}^{p}|\cdot |\Delta _{n-1}^{p}|\cdot |\Delta _{n}^{p}|.
\end{eqnarray*}

$11)$ Combine (\ref{kd}) and (\ref{h3n+1}), we have
\begin{eqnarray*}
\left\vert P^{t}\Delta _{3n+1}^{3p}P\right\vert  &\equiv &\left\vert
\begin{array}{ccc}
-\Gamma _{n+1}^{p} & \left( \Gamma _{n+1}^{p+1}\right) ^{(n+1)} & \left(
\Gamma _{n+1}^{p}\right) ^{(n+1)} \\
\left( \Gamma _{n+1}^{p+1}\right) _{(n+1)} & \Gamma _{n}^{p} & -\Gamma
_{n}^{p+1} \\
\left( \Gamma _{n+1}^{p}\right) _{(n+1)} & -\Gamma _{n}^{p+1} & \Gamma
_{n}^{p+2}%
\end{array}%
\right\vert  \\
&=&\left\vert
\begin{array}{ccc}
-\Gamma _{n+1}^{p} & \left( \Gamma _{n+1}^{p+1}\right) ^{(n+1)} &
\boldsymbol{0}_{(n+1)\times n} \\
\left( \Gamma _{n+1}^{p+1}\right) _{(n+1)} & \Gamma _{n}^{p} & \boldsymbol{0}%
_{n\times n} \\
\left( \Gamma _{n+1}^{p}\right) _{(n+1)} & -\Gamma _{n}^{p+1} & \Gamma
_{n}^{p}+\Gamma _{n}^{p+2}%
\end{array}%
\right\vert  \\
&=&\left\vert \Gamma _{n}^{p}+\Gamma _{n}^{p+2}\right\vert \cdot \left\vert
\begin{array}{cc}
-\Gamma _{n+1}^{p} & \left( \Gamma _{n+1}^{p+1}\right) ^{(n+1)} \\
\left( \Gamma _{n+1}^{p+1}\right) _{(n+1)} & \Gamma _{n}^{p}%
\end{array}%
\right\vert  \\
&=&\left\vert \Gamma _{n}^{p}+\Gamma _{n}^{p+2}\right\vert \cdot \left\vert
\begin{array}{cc}
-\Gamma _{n+1}^{p} & \boldsymbol{0}_{(n+1)\times n} \\
\left( \Gamma _{n+1}^{p+1}\right) _{(n+1)} & \Gamma _{n}^{p}+\Gamma
_{n}^{p+2}%
\end{array}%
\right\vert .
\end{eqnarray*}%
Hence,
\begin{equation*}
\left\vert \Delta _{3n+1}^{3p}\right\vert \equiv (-1)^{n+1}\left\vert \Gamma
_{n+1}^{p}\right\vert \cdot \left\vert \Delta _{n}^{p}\right\vert ^{2}.
\end{equation*}

$12)$ Combine (\ref{kd}) and (\ref{h3n+2}), we have
\begin{eqnarray*}
\left\vert P^{t}\Delta _{3n+2}^{3p}P\right\vert  &\equiv &\left\vert
\begin{array}{ccc}
-\Gamma _{n+1}^{p} & \Gamma _{n+1}^{p+1} & \left( \Gamma _{n+1}^{p}\right)
^{(n+1)} \\
\Gamma _{n+1}^{p+1} & \Gamma _{n+1}^{p} & -\left( \Gamma _{n+1}^{p+1}\right)
^{(n+1)} \\
\left( \Gamma _{n+1}^{p}\right) _{(n+1)} & -\left( \Gamma
_{n+1}^{p+1}\right) _{(n+1)} & \Gamma _{n}^{p+2}%
\end{array}%
\right\vert  \\
&=&\left\vert
\begin{array}{ccc}
-\Gamma _{n+1}^{p} & \Gamma _{n+1}^{p+1} & \boldsymbol{0}_{(n+1)\times n} \\
\Gamma _{n+1}^{p+1} & \Gamma _{n+1}^{p} & \boldsymbol{0}_{(n+1)\times n} \\
\left( \Gamma _{n+1}^{p}\right) _{(n+1)} & -\left( \Gamma
_{n+1}^{p+1}\right) _{(n+1)} & \Gamma _{n}^{p}+\Gamma _{n}^{p+2}%
\end{array}%
\right\vert .
\end{eqnarray*}

By Lemma \ref{lem:matirx},
\begin{eqnarray*}
\left\vert \Delta _{3n+2}^{3p}\right\vert  &\equiv &\left\vert \Delta
_{n}^{p}\right\vert \cdot \left\vert
\begin{array}{cc}
-\Gamma _{n+1}^{p} & \Gamma _{n+1}^{p+1} \\
\Gamma _{n+1}^{p+1} & \Gamma _{n+1}^{p}%
\end{array}%
\right\vert  \\
&=&(-1)^{n}\left\vert \Gamma _{n+2}^{p}\right\vert \cdot \left\vert \Delta
_{n}^{p}\right\vert ^{2}+(-1)^{n+1}\left\vert \Gamma _{n+1}^{p}\right\vert
\cdot \left\vert \Delta _{n}^{p}\right\vert \cdot \left\vert \Delta
_{n+1}^{p}\right\vert .
\end{eqnarray*}

$13)$ Combine (\ref{kd}) and (\ref{h3n}), we have
\begin{eqnarray*}
\left\vert P^{t}\Delta _{3n}^{3p+1}P\right\vert  &\equiv &\left\vert
\begin{matrix}
\Gamma _{n}^{p+1} & \Gamma _{n}^{p} & \mathbf{-}\Gamma _{n}^{p+1} \\
\Gamma _{n}^{p} & \mathbf{-}\Gamma _{n}^{p+1} & \Gamma _{n}^{p+2} \\
\mathbf{-}\Gamma _{n}^{p+1} & \Gamma _{n}^{p+2} & \Gamma _{n}^{p+1}%
\end{matrix}%
\right\vert  \\
&=&\left\vert
\begin{matrix}
\mathbf{0}_{n\times n} & \Gamma _{n}^{p}+\Gamma _{n}^{p+2} & \mathbf{0}%
_{n\times n} \\
\Gamma _{n}^{p} & \mathbf{-}\Gamma _{n}^{p+1} & \Gamma _{n}^{p+2} \\
\mathbf{-}\Gamma _{n}^{p+1} & \Gamma _{n}^{p+2} & \Gamma _{n}^{p+1}%
\end{matrix}%
\right\vert  \\
&=&(-1)^{n}\left\vert \Gamma _{n}^{p}+\Gamma _{n}^{p+2}\right\vert \cdot
\left\vert
\begin{matrix}
\Gamma _{n}^{p} & \Gamma _{n}^{p+2} \\
\mathbf{-}\Gamma _{n}^{p+1} & \Gamma _{n}^{p+1}%
\end{matrix}%
\right\vert  \\
&=&(-1)^{n}\left\vert \Gamma _{n}^{p}+\Gamma _{n}^{p+2}\right\vert \cdot
\left\vert
\begin{matrix}
\Gamma _{n}^{p}+\Gamma _{n}^{p+2} & \Gamma _{n}^{p+2} \\
\mathbf{0}_{n\times n} & \Gamma _{n}^{p+1}%
\end{matrix}%
\right\vert  \\
&=&(-1)^{n}\left\vert \Gamma _{n}^{p}+\Gamma _{n}^{p+2}\right\vert ^{2}\cdot
\left\vert \Gamma _{n}^{p+1}\right\vert .
\end{eqnarray*}%
Therefore,
\begin{equation*}
\left\vert \Delta _{3n}^{3p+1}\right\vert \equiv (-1)^{n}\left\vert \Gamma
_{n}^{p+1}\right\vert \cdot \left\vert \Delta _{n}^{p}\right\vert ^{2}.
\end{equation*}

$14)$ Combine (\ref{kd}) and (\ref{h3n+1}), we have
\begin{eqnarray*}
\left\vert P^{t}\Delta _{3n+1}^{3p+1}P\right\vert  &\equiv &\left\vert
\begin{array}{ccc}
\Gamma _{n+1}^{p+1} & \left( \Gamma _{n+1}^{p}\right) ^{(n+1)} & -\left(
\Gamma _{n+1}^{p+1}\right) ^{(n+1)} \\
\left( \Gamma _{n+1}^{p}\right) _{(n+1)} & -\Gamma _{n}^{p+1} & \Gamma
_{n}^{p+2} \\
-\left( \Gamma _{n+1}^{p+1}\right) _{(n+1)} & \Gamma _{n}^{p+2} & \Gamma
_{n}^{p+1}%
\end{array}%
\right\vert  \\
&=&\left\vert
\begin{array}{ccc}
\Gamma _{n+1}^{p+1} & \left( \Gamma _{n+1}^{p}\right) ^{(n+1)} & \mathbf{0}%
_{(n+1)\times n} \\
\left( \Gamma _{n+1}^{p}\right) _{(n+1)} & -\Gamma _{n}^{p+1} & \Gamma
_{n}^{p}+\Gamma _{n}^{p+2} \\
-\left( \Gamma _{n+1}^{p+1}\right) _{(n+1)} & \Gamma _{n}^{p+2} & \mathbf{0}%
_{n\times n}%
\end{array}%
\right\vert  \\
&=&(-1)^{n}\left\vert \Gamma _{n}^{p}+\Gamma _{n}^{p+2}\right\vert \cdot
\left\vert
\begin{array}{cc}
\Gamma _{n+1}^{p+1} & \left( \Gamma _{n+1}^{p}\right) ^{(n+1)} \\
-\left( \Gamma _{n+1}^{p+1}\right) _{(n+1)} & \Gamma _{n}^{p+2}%
\end{array}%
\right\vert  \\
&=&(-1)^{n}\left\vert \Gamma _{n}^{p}+\Gamma _{n}^{p+2}\right\vert \cdot
\left\vert
\begin{array}{cc}
\Gamma _{n+1}^{p+1} & \left( \Gamma _{n+1}^{p}\right) ^{(n+1)} \\
0_{n\times (n+1)} & \Gamma _{n}^{p}+\Gamma _{n}^{p+2}%
\end{array}%
\right\vert .
\end{eqnarray*}%
Hence,
\begin{equation*}
\left\vert \Delta _{3n+1}^{3p+1}\right\vert \equiv (-1)^{n}\left\vert \Gamma
_{n+1}^{p+1}\right\vert \cdot \left\vert \Delta _{n}^{p}\right\vert ^{2}.
\end{equation*}

$15)$ Combine (\ref{kd}) and (\ref{h3n+2}), we have
\begin{eqnarray*}
\left\vert P^{t}\Delta _{3n+2}^{3p+1}P\right\vert  &\equiv &\left\vert
\begin{array}{ccc}
\Gamma _{n+1}^{p+1} & \Gamma _{n+1}^{p} & -\left( \Gamma _{n+1}^{p+1}\right)
^{(n+1)} \\
\Gamma _{n+1}^{p} & -\Gamma _{n+1}^{p+1} & \left( \Gamma _{n+1}^{p+2}\right)
^{(n+1)} \\
-\left( \Gamma _{n+1}^{p+1}\right) _{(n+1)} & \left( \Gamma
_{n+1}^{p+2}\right) _{(n+1)} & \Gamma _{n}^{p+1}%
\end{array}%
\right\vert  \\
&=&\left\vert
\begin{array}{ccc}
\Gamma _{n+1}^{p+1} & \Gamma _{n+1}^{p} & \mathbf{0}_{(n+1)\times n} \\
\Gamma _{n+1}^{p} & -\Gamma _{n+1}^{p+1} & \mathbf{0}_{(n+1)\times n} \\
-\left( \Gamma _{n+1}^{p+1}\right) _{(n+1)} & \left( \Gamma
_{n+1}^{p+2}\right) _{(n+1)} & \Gamma _{n}^{p+1}+\Gamma _{n}^{p+3}%
\end{array}%
\right\vert  \\
&=&\left\vert \Gamma _{n}^{p+1}+\Gamma _{n}^{p+3}\right\vert \cdot
\left\vert
\begin{array}{cc}
\Gamma _{n+1}^{p+1} & \Gamma _{n+1}^{p} \\
\Gamma _{n+1}^{p} & -\Gamma _{n+1}^{p+1}%
\end{array}%
\right\vert .
\end{eqnarray*}%
By Lemma \ref{lem:matirx},
\begin{equation*}
\left\vert \Delta _{3n+2}^{3p+1}\right\vert \equiv (-1)^{n}\left\vert \Gamma
_{n+2}^{p}\right\vert \cdot \left\vert \Delta _{n}^{p}\right\vert \cdot
\left\vert \Delta _{n}^{p+1}\right\vert +(-1)^{n+1}\left\vert \Gamma
_{n+1}^{p}\right\vert \cdot \left\vert \Delta _{n}^{p+1}\right\vert \cdot
\left\vert \Delta _{n+1}^{p}\right\vert .
\end{equation*}

$16)$ Combine (\ref{kd}) and (\ref{h3n}), we have
\begin{eqnarray*}
\left\vert P^{t}\Delta _{3n}^{3p+2}P\right\vert  &\equiv &\left\vert
\begin{matrix}
\Gamma _{n}^{p} & -\Gamma _{n}^{p+1} & \Gamma _{n}^{p+2} \\
-\Gamma _{n}^{p+1} & \Gamma _{n}^{p+2} & \Gamma _{n}^{p+1} \\
\Gamma _{n}^{p+2} & \Gamma _{n}^{p+1} & -\Gamma _{n}^{p+2}%
\end{matrix}%
\right\vert  \\
&=&\left\vert
\begin{matrix}
\Gamma _{n}^{p}+\Gamma _{n}^{p+2} & \mathbf{0}_{n\times n} & \mathbf{0}%
_{n\times n} \\
-\Gamma _{n}^{p+1} & \Gamma _{n}^{p+2} & \Gamma _{n}^{p+1} \\
\Gamma _{n}^{p+2} & \Gamma _{n}^{p+1} & -\Gamma _{n}^{p+2}%
\end{matrix}%
\right\vert  \\
&=&\left\vert \Gamma _{n}^{p}+\Gamma _{n}^{p+2}\right\vert \cdot \left\vert
\begin{matrix}
\Gamma _{n}^{p+2} & \Gamma _{n}^{p+1} \\
\Gamma _{n}^{p+1} & -\Gamma _{n}^{p+2}%
\end{matrix}%
\right\vert .
\end{eqnarray*}%
By Lemma \ref{lem:matirx},
\begin{equation*}
\left\vert \Delta _{3n}^{3p+2}\right\vert \equiv (-1)^{n}\left\vert \Gamma
_{n}^{p+1}\right\vert \cdot \left\vert \Delta _{n}^{p}\right\vert \cdot
\left\vert \Delta _{n}^{p+1}\right\vert +(-1)^{n+1}\left\vert \Gamma
_{n+1}^{p+1}\right\vert \cdot \left\vert \Delta _{n}^{p}\right\vert \cdot
\left\vert \Delta _{n-1}^{p+1}\right\vert .
\end{equation*}

$17)$ Combine (\ref{kd}) and (\ref{h3n+1}), we have
\begin{eqnarray*}
\left\vert P^{t}\Delta _{3n+1}^{3p+2}P\right\vert  &\equiv &\left\vert
\begin{matrix}
\Gamma _{n+1}^{p} & -\left( \Gamma _{n+1}^{p+1}\right) ^{(n+1)} & \left(
\Gamma _{n+1}^{p+2}\right) ^{(n+1)} \\
-\left( \Gamma _{n+1}^{p+1}\right) _{(n+1)} & \Gamma _{n}^{p+2} & \Gamma
_{n}^{p+1} \\
\left( \Gamma _{n+1}^{p+2}\right) _{(n+1)} & \Gamma _{n}^{p+1} & -\Gamma
_{n}^{p+2}%
\end{matrix}%
\right\vert  \\
&=&\left\vert
\begin{matrix}
\Gamma _{n+1}^{p} & \mathbf{0}_{(n+1)\times n} & \left( \Gamma
_{n+1}^{p+2}\right) ^{(n+1)} \\
-\left( \Gamma _{n+1}^{p+1}\right) _{(n+1)} & \mathbf{0}_{n\times n} &
\Gamma _{n}^{p+1} \\
\left( \Gamma _{n+1}^{p+2}\right) _{(n+1)} & \Gamma _{n}^{p+1}+\Gamma
_{n}^{p+3} & -\Gamma _{n}^{p+2}%
\end{matrix}%
\right\vert  \\
&=&(-1)^{n}\left\vert \Gamma _{n}^{p+1}+\Gamma _{n}^{p+3}\right\vert \cdot
\left\vert
\begin{matrix}
\Gamma _{n+1}^{p} & \left( \Gamma _{n+1}^{p+2}\right) ^{(n+1)} \\
-\left( \Gamma _{n+1}^{p+1}\right) _{(n+1)} & \Gamma _{n}^{p+1}%
\end{matrix}%
\right\vert  \\
&=&(-1)^{n}\left\vert \Gamma _{n}^{p+1}+\Gamma _{n}^{p+3}\right\vert \cdot
\left\vert
\begin{matrix}
\Gamma _{n+1}^{p} & \left( \Gamma _{n+1}^{p+2}\right) ^{(n+1)} \\
0_{n\times (n+1)} & \Gamma _{n}^{p+1}+\Gamma _{n}^{p+3}%
\end{matrix}%
\right\vert .
\end{eqnarray*}%
Hence,
\begin{equation*}
\left\vert \Delta _{3n+1}^{3p+2}\right\vert \equiv (-1)^{n}\left\vert \Gamma
_{n+1}^{p}\right\vert \cdot \left\vert \Delta _{n}^{p+1}\right\vert ^{2}.
\end{equation*}

$18)$ Combine (\ref{kd}) and (\ref{h3n+2}), we have
\begin{eqnarray*}
\left\vert P^{t}\Delta _{3n+2}^{3p+2}P\right\vert  &\equiv &\left\vert
\begin{matrix}
\Gamma _{n+1}^{p} & -\Gamma _{n+1}^{p+1} & \left( \Gamma _{n+1}^{p+2}\right)
^{(n+1)} \\
-\Gamma _{n+1}^{p+1} & \Gamma _{n+1}^{p+2} & \left( \Gamma
_{n+1}^{p+1}\right) ^{(n+1)} \\
\left( \Gamma _{n+1}^{p+2}\right) _{(n+1)} & \left( \Gamma
_{n+1}^{p+1}\right) _{(n+1)} & -\Gamma _{n}^{p+2}%
\end{matrix}%
\right\vert  \\
&=&\left\vert
\begin{matrix}
\Gamma _{n+1}^{p} & -\Gamma _{n+1}^{p+1} & \mathbf{0}_{(n+1)\times n} \\
-\Gamma _{n+1}^{p+1} & \Gamma _{n+1}^{p+2} & \left( \Gamma
_{n+1}^{p+1}+\Gamma _{n+1}^{p+3}\right) ^{(n+1)} \\
\left( \Gamma _{n+1}^{p+2}\right) _{(n+1)} & \left( \Gamma
_{n+1}^{p+1}\right) _{(n+1)} & \mathbf{0}_{n\times n}%
\end{matrix}%
\right\vert  \\
&=&\left\vert
\begin{matrix}
\Gamma _{n+1}^{p} & -\Gamma _{n+1}^{p+1} & \mathbf{0}_{(n+1)\times n} \\
\begin{array}{c}
\mathbf{0}_{n\times (n+1)} \\
-\left( \alpha _{p+n+1}^{n+1}\right) ^{t}%
\end{array}
&
\begin{array}{c}
\mathbf{0}_{n\times (n+1)} \\
\left( \alpha _{p+n+2}^{n+1}\right) ^{t}%
\end{array}
& \left( \Gamma _{n+1}^{p+1}+\Gamma _{n+1}^{p+3}\right) ^{(n+1)} \\
\left( \Gamma _{n+1}^{p+2}\right) _{(n+1)} & \left( \Gamma
_{n+1}^{p+1}\right) _{(n+1)} & \mathbf{0}_{n\times n}%
\end{matrix}%
\right\vert  \\
&=&\left\vert \Gamma _{n}^{p+1}+\Gamma _{n}^{p+3}\right\vert \cdot
\left\vert
\begin{matrix}
\Gamma _{n+1}^{p} & -\Gamma _{n+1}^{p+1} \\
-\left( \alpha _{p+n+1}^{n+1}\right) ^{t} & \left( \alpha
_{p+n+2}^{n+1}\right) ^{t} \\
\left( \Gamma _{n+1}^{p+2}\right) _{(n+1)} & \left( \Gamma
_{n+1}^{p+1}\right) _{(n+1)}%
\end{matrix}%
\right\vert  \\
&=&\left\vert \Gamma _{n}^{p+1}+\Gamma _{n}^{p+3}\right\vert \cdot
(-1)\left\vert
\begin{matrix}
\left( \Gamma _{n+2}^{p}\right) ^{(n+2)} & -\left( \Gamma
_{n+2}^{p+1}\right) ^{(n+2)} \\
\left( \Gamma _{n+1}^{p+2}\right) _{(n+1)} & \left( \Gamma
_{n+1}^{p+1}\right) _{(n+1)}%
\end{matrix}%
\right\vert .
\end{eqnarray*}%
Note that
\begin{eqnarray*}
&&\left\vert
\begin{matrix}
\left( \Gamma _{n+2}^{p}\right) ^{(n+2)} & -\left( \Gamma
_{n+2}^{p+1}\right) ^{(n+2)} \\
\left( \Gamma _{n+1}^{p+2}\right) _{(n+1)} & \left( \Gamma
_{n+1}^{p+1}\right) _{(n+1)}%
\end{matrix}%
\right\vert  \\
&=&\left\vert
\begin{array}{ccc}
\left( \Gamma _{n+2}^{p}\right) ^{(n+2)} & \mathbf{0}_{(n+2)\times n} &
-\alpha _{p+n+1}^{n+2} \\
\left( \Gamma _{n+1}^{p+2}\right) _{(n+1)} & \Gamma _{n}^{p+1}+\Gamma
_{n}^{p+3} & \alpha _{p+n+1}^{n}%
\end{array}%
\right\vert  \\
&=&(-1)^{n}\left\vert \Gamma _{n}^{p+1}+\Gamma _{n}^{p+3}\right\vert \cdot
(-1)\left\vert \Gamma _{n+2}^{p}\right\vert ,
\end{eqnarray*}%
we have
\begin{equation*}
\left\vert \Delta _{3n+2}^{3p+2}\right\vert \equiv (-1)^{n}\left\vert \Gamma
_{n+2}^{p}\right\vert \cdot \left\vert \Delta _{n}^{p+1}\right\vert ^{2}.
\end{equation*}
\end{proof}

\end{document}